\documentclass[12pt]{amsart}
\usepackage{fullpage}
\usepackage[english]{babel}
\usepackage{mathtools}
\usepackage{amssymb}
\usepackage{mathrsfs}
\usepackage{nicefrac}
\usepackage{upgreek}
\usepackage[T1]{fontenc}
\usepackage{graphicx}
\usepackage{stmaryrd}
\usepackage{graphicx}
\usepackage{booktabs}
\usepackage{hyperref}

\usepackage[all]{xy}
\usepackage{graphics}
\usepackage{color}
\usepackage[T1]{fontenc}

\usepackage{enumitem}
\usepackage[table]{xcolor}
\usepackage{colortbl} 
\usepackage{diagbox} 
\usepackage{array}   
\usepackage{makecell} 

\makeatletter
\newtheorem*{rep@theorem}{\rep@title}
\newcommand{\newreptheorem}[2]{%
\newenvironment{rep#1}[1]{%
 \def\rep@title{#2 \ref{##1}}%
 \begin{rep@theorem}}%
 {\end{rep@theorem}}}
\makeatother

\newcommand{\C}[1]{{\mathcal #1}}

\newtheorem{theorem}{Theorem}[section]
\newtheorem{lemma}[theorem]{Lemma}
\newtheorem{claim}{Claim}
\newtheorem{corollary}[theorem]{Corollary}
\newtheorem{proposition}[theorem]{Proposition}
\newtheorem{conjecture}[theorem]{Conjecture}
\theoremstyle{definition}\newtheorem{definition}[theorem]{Definition}
\theoremstyle{remark}\newtheorem{remark}[theorem]{Remark}
\theoremstyle{definition}
\makeatletter\@addtoreset{case}{example}\makeatother

\begin{document}

\title{Stability, bounded generation and strong boundedness}

\author{Alexander A. Trost}
\address{Institute of Mathematics, University of Wroc{\l}aw, 50-383 Wroc{\l}aw, Poland}
\email{alexander.trost@math.uni.wroc.pl, trostalois@gmail.com}

\maketitle

\begin{abstract}
We provide bounds linear in the rank for the generalized conjugacy diameters, introduced by Kedra, Libman and Martin, for the special linear and symplectic groups defined over the rings of integers of global fields by way of using certain stability considerations familiar from classical algebraic K-theory. This determines the growth rates of these generalized conjugacy diameters.
\end{abstract}

\section{Introduction}

For finite groups the speed of generation by non-trivial conjugacy classes for finite simple groups, has been extensively investigated under the term of covering numbers \cite{MR783068,MR1865975}. Similar considerations for non-finite, non-simple groups lead to the study of conjugation-invariant word norms on groups first introduced in \cite{MR2509711}. If all conjugation-invariant norms on a group $G$ have finite diameters, this group $G$ is called \textit{bounded}. Conjugation-invariant norms on arithmetic Chevalley groups (more specifically ${\rm SL}_{n\geq 3}(\mathbb{Z})$) were one of the first non-trivial examples considered in \cite[Example~1.1]{MR2509711} where it was shown that ${\rm SL}_{n\geq 3}(\mathbb{Z})$ is bounded. 

A more precise question concerning conjugation-invariant norms is how the diameter of conjugation-invariant word norms $\|\cdot\|_T$ induced by finitely many, generating conjugacy classes $T$ of ${\rm SL}_{n\geq 3}(\mathbb{Z})$ depend on the set $T:$ One of the first such results was proven by Morris \cite{MR2357719} and Kedra, Libman and Martin \cite{KLM} and shows that ${\rm SL}_{n\geq 3}(\mathbb{Z})$ for example satisfy the stronger property of being \textit{strongly bounded}; that is for a group $G$ and each natural number $k$, there is a bound $\Delta_k(G)\in\mathbb{N}$ such that each word norm $\|\cdot\|_T$ given by $k$ generating conjugacy classes $T$ of $G$ has diameter at most $\Delta_k(G)$: 

\begin{theorem}\cite[Corollary~6.2]{KLM}\label{old_KLM_thm}
Let $R:=\mathscr{O}_K^S$ be the ring of S-algebraic integers in a number field whose class number is one and $n\geq 3.$ Then ${\rm SL}_n(R)$ is normally generated by any elementary matrix $E_{ij}(1)$ and $k\leq\Delta_k({\rm SL}_n(R))\leq (4n+51)(4n+4)k$ holds for all $k\in\mathbb{N}.$ 
\end{theorem}

Theorem~\ref{old_KLM_thm} however immediately raised several question: First and perhaps most obviously, the restriction to the special linear group over other natural matrix groups like the symplectic groups begs the question whether strong boundedness might also hold for those other matrix groups. Second, the restriction on the ring of algebraic integers $R$ to be a principal ideal domain, seems artificial given that the rings of all S-algebraic integers in most number fields aren't principal ideal domains. Related to this, there is the question as to which other rings besides rings of integers in number fields might satisfy results like Theorem~\ref{old_KLM_thm}. Third and most subtly, there is the question on the growth behaviour of $(\frac{\Delta_k({\rm SL}_n(R))}{k})_n$ in $n$ for a fixed value of $k.$ The first question gave way relatively quickly: For example, \cite[Theorem~3]{explicit_strong_bound_sp_2n} showed that essentially the same methods as the ones used to prove Theorem~\ref{old_KLM_thm} also produced bounds on $\Delta_k$ for higher rank symplectic groups ${\rm Sp}_{2n}(R)$ where $n\geq 3$ (and $R$ as in Theorem~\ref{old_KLM_thm}). Further, the result \cite[Theorem~5.13]{General_strong_bound} showed (by using more classical K-theoretic results) that for $R$ the ring of all S-algebraic integers in any number field, in fact all so-called split Chevalley groups $G(\Phi,R)$ (a certain type of matrix groups encompassing among others the special linear and symplectic groups) are also strongly bounded. 
As just mentioned, the result \cite[Theorem~5.13]{General_strong_bound} also addressed the second question: It is not necessary for the ring of integers to be a principal ideal domain; the ring of all S-algebraic integers in any number field works. In fact, as we showed in \cite[Theorem~1.5]{Chevalley_positive_char_tentative}, strong boundedness also holds if the ring of all integers in a number field is replaced with the ring of all integers in a global field (a larger class of fields encompassing not only number fields but also global function fields). As an aside: Another class of rings $R$ for which $G(\Phi,R)$ for $\Phi\neq A_1,C_2,G_2$ satisfies strong boundedness are semi-local rings \cite[Theorem~4.8]{General_strong_bound}.
The third question however, proved to be quite a bit more intractable: A more careful reading of the proof of Theorem~\ref{old_KLM_thm} shows that $\frac{\Delta_k(G(\Phi,R))}{k}\geq {\rm rank}(\Phi)$ holds for all $k\in\mathbb{N}$ and root systems $\Phi$ \cite[Theorem~3]{explicit_strong_bound_sp_2n}(besides $C_2$ and $G_2$, more on that below). However, the upper bounds on $\frac{\Delta_k(G(\Phi,R))}{k}$ were far from this lower bound in all prior papers \cite{KLM,Chevalley_positive_char_tentative,explicit_strong_bound_sp_2n,General_strong_bound}: They are either non-explicit \cite{Chevalley_positive_char_tentative,General_strong_bound} or don't agree with the known lower bounds \cite{KLM,explicit_strong_bound_sp_2n} not even in regards to their growth: As seen for example in Theorem~\ref{old_KLM_thm}, the upper bound on $\frac{\Delta_k({\rm SL}_n(R))}{k}\leq (4n+51)(4n+4)$ is quadratic in $n$. In this paper, we will finally solve this problem of the growth rate (for some Chevalley groups $G(\Phi,R)$ and $R$ the ring of all integers in any global field) by showing that there are also upper bounds on $\frac{\Delta_k(G(\Phi,R))}{k}$ which are linear in ${\rm rank}(\Phi):$

\begin{theorem}\label{main_thm0}
Let $K$ be a global field and $S$ a finite non-empty set of valuations of $K$ containing all archimedean valuations of $K$. Let $R:=\mathscr{O}_K^S$ be the ring of all S-algebraic integers of $K.$ Further, let  $\Phi\in\{A_{n\geq 2},C_{n\geq 2},E_6,E_7,E_8,F_4,G_2\}$ be given. Then there is a constant $C\in\mathbb{N}$ independent of $\Phi,K$ and $S$ and a constant $D(K)$ independent of $\Phi$ and $S$ but depending on the field $K$, such that for all natural numbers $k$ with 
$$
k\geq 
\begin{cases}
|\{\C P\unlhd R\mid R/\C P=\mathbb{F}_2\}|&\text{, if }\Phi=C_2,G_2\\
1&\text{, else,}
\end{cases}
$$
the following chain of inequalities holds: 
$$
k\cdot{\rm rank}(\Phi)\leq\Delta_k(G(\Phi,R))\leq (D(K)+Ck)\cdot{\rm rank}(\Phi).
$$
In particular, if $k\geq D(K)$ also holds, then $k\cdot{\rm rank}(\Phi)\leq\Delta_k(G(\Phi,R))\leq (C+1)k\cdot{\rm rank}(\Phi)$ holds. 
\end{theorem}

The restriction on the natural number $k$ to be at least $|\{\C P\unlhd\mid R/\C P=\mathbb{F}_2\}|=:r(R)$, if $\Phi=C_2$ or $G_2$, is necessary: Namely, as seen in \cite[Theorem~6.3]{General_strong_bound}, the number $r(R)$ is the smallest size a generating set of conjugacy classes of ${\rm Sp}_4(R)$ and $G_2(R)$ can have (and there are indeed generating sets of conjugacy classes of this size). Here, we should also briefly discuss the subtle difference of Theorem~\ref{main_thm0} for the specific root systems $\Phi=C_2$ and $G_2$ to earlier results for those root systems: Earlier results of strong boundedness for ${\rm Sp}_4(R)$ and $G_2(R)$ (like \cite[Theorem~5.13]{General_strong_bound}) only showed that there is a constant $C(\Phi,R)\in\mathbb{N}$ such that
$$
\Delta_k(G(\Phi,R))\leq C(\Phi,R)\cdot k
$$
holds for all $k\in\mathbb{N}$ and $\Phi=C_2,G_2.$ In particular, the dependency on the field $K$ is contained in the proportionality factor of $k$. In contrast, in Theorem~\ref{main_thm0} the proportionality factor of $k$ is independent of the global field $K$ and the "number theory" of $K$ is entirely contained in an additional constant term. As an aside: The set $\{\C P\unlhd R\mid R/\C P=\mathbb{F}_2\}$ in Theorem~\ref{main_thm0} is always finite. We leave this as an excercise to the reader but note that the global function field case uses that there are only finitely many primes of each degree \cite[Theorem~5.12]{MR1876657}. The constant $C$ in Theorem~\ref{main_thm0} is usually explicit and the constant $D(K)$ depends on the global field in an (usually) explicit manner. However, to what extent these constants are known (and in some cases even present), is quite dependent on the specific $K,S$ and $\Phi$ and we postpone this discussion till after the proof of Theorem~\ref{main_thm0}. 

The main difference between the proof of Theorem~\ref{main_thm0} and earlier results like Theorem~\ref{old_KLM_thm} and the results from \cite{Chevalley_positive_char_tentative,explicit_strong_bound_sp_2n,General_strong_bound} is a difference in viewpoint: In the earlier papers, one proceeds along a two-step process: One first produces from a given collection of generating conjugacy classes $T$ of $G(\Phi,\mathscr{O}_K^S)$ a specific generating set (so-called root elements), and then one applies bounded generation results for these root elements. This two-step strategy is somewhat inefficient; the first step needs linear in ${\rm rank}(\Phi)\cdot|T|$ many elements from $T$ and the second step then requires ${\rm rank}(\Phi)$ many conjugates of root elements. Together this results in the quadratic bounds as in Theorem~\ref{old_KLM_thm}.

To simplify slightly, in our strategy, we instead produce from a single non-central element $B$ of $T$, a congruence subgroup $C(\Phi,I)$ of $G(\Phi,R)$ sitting in a $\|\cdot\|_T$-ball of radius linear in ${\rm rank}(\Phi)$ (already this step is non-trivial). We then use this congruence subgroup to transform the problem of understanding $\Delta_k(G(\Phi,R))$ into the problem of understanding $\Delta_k(G(\Phi,R_I))$ for a semi-local ring $R_I$ arising essentially as the $I$-adic completion of $R$ following an idea similar to one in \cite{avni_meiri_2019}. For such semi-local rings the $\Delta_k$ are much better understood from an asymptotic view point (\cite[Theorem~6.3]{KLM} and \cite[Theorem~2]{explicit_strong_bound_sp_2n}) and this fact together with a simple commutator trick to treat the direct factors of $R_I$ all at once proves Theorem~\ref{main_thm0}. Along the way of the proof, we will also prove the following:

\begin{theorem}\label{main_thm1}
Let $F$ be an algebraic extension of the finite field $\mathbb{F}$, $C$ an irreducible, nonsingular, geometrically integral, projective curve defined over $F$, $U$ a proper, non-empty Zariski-open set in $C$ defined over $F$. Further, let $R:=F_C[U]=\mathscr{O}_C(U)$ be the ring of $F$-defined, regular functions on $U$ and $\Phi\in\{A_{n\geq 2},C_{n\geq 2},E_6,E_7,E_8,F_4,G_2\}$. Then there is a constant $D\in\mathbb{N}$ such that for all $k\in\mathbb{N}$, one has $\Delta_k\left(G(\Phi,R)\right)\leq D\cdot{\rm rank}(\Phi)\cdot k$
and this constant $D$ does not depend on $F,C,U$ or $\Phi.$ Further, for $\Phi=A_{n\geq 3},C_{n\geq 4},$ one can choose $D$ as less than $1064062.$ 
\end{theorem}

Lastly, one of the main ingredients (and the one used to actually control the diameter of $C(\Phi,I)$ with respect to $\|\cdot\|_T$ in the proof of Theorem~\ref{main_thm0}) is a more general version of the following result:

\begin{theorem}\label{main_thm2}
Let $R$ and $\Phi$ be as in Theorem~\ref{main_thm1} and let $I$ be an ideal in $R$. Then there is a constant $L\in\mathbb{N}$ such that each element of $\bar{E}(\Phi,I)$ can be written as a product of $L\cdot{\rm rank}(\Phi)$ many elements from the set $\left\{A\varepsilon_{\phi}(x)A^{-1}\mid A\in G(\Phi,R),x\in I,\phi\in\Phi\right\}$ and this constant $L$ does not depend on $F,C,U$ or $I$.
\end{theorem}

The root elements $\varepsilon_{\phi}(x)$ (and the group $\bar{E}(\Phi,I)$) are discussed in Subsection~\ref{chevalley_groups_def}, but one can think of root elements as generalizations of elementary matrices in ${\rm SL}_n(R)$ for other matrix groups. Then the group $\bar{E}(\Phi,I)$ is the subgroup of $G(\Phi,R)$ generated by $G(\Phi,R)$-conjugates of elements of the set $\left\{\varepsilon_{\phi}(x)\mid x\in I,\phi\in\Phi\right\}$.

The structure of the paper is straight-forward: The second section introduces most of the notations concerning Chevalley groups, conjugation-invariant norms and local and global fields. The third section proves Theorem~\ref{main_thm2}. The fourth section finally proves Theorem~\ref{main_thm0} and Theorem~\ref{main_thm1} and the last section contains a brief discussion of the value of $D(K)$ and of some related problems and conjectures.

\section*{Acknowledgements}

The main idea for this preprint grew in a roundabout way out of a different (ultimately unsuccessful) project conceived together with Benjamin Martin and I want to express my gratitude to him. I also want to thank the anonymous reviewer whose insightful recommendations and careful reading of the paper improved its quality significantly.

\section{Basic definitions and notions}\label{definitions}

\subsection{Chevalley groups and congruence subgroups}\label{chevalley_groups_def}

The main subject of this paper are simply-connected split Chevalley groups defined over certain rings of integers of global (and local) fields and conjugation-invariant word norms defined on these groups. Chevalley groups $G(\Phi,\cdot)$ are certain algebraic groups (cf. \cite{MR3616493,General_strong_bound} for more details).

For $\Phi$ an irreducible root system and $R$ a commutative ring with $1,$ the simply-connected Chevalley group $G(\Phi,R)$ contains the so-called \textit{root elements} modulo a choice of a maximal torus in $G(\Phi,\mathbb{C})$: For each $\phi\in\Phi$ and $x\in R$, there are elements $\varepsilon_{\phi}(x)\in G(\Phi,R)$. These root elements satisfy a couple of commutator formulas (called \textit{Steinberg commutator relations}, cf. \cite[Proposition~33.3-5]{MR0396773}). We further need two different congruence subgroups of $G(\Phi,R)$:

\begin{definition}
Let $R$ be a commutative ring with $1$, $I$ a proper non-zero ideal of $R$ and $\Phi$ an irreducible root system. Than the kernel $C(\Phi,I)$ of the group homomorphism $G(\Phi,R)\to G(\Phi,R/I)$ induced by the quotient map $R\to R/I$ is called the \emph{principal $I$-congruence subgroup of $G(\Phi,R)$}. The subgroup $\bar{E}(\Phi,I)$ of $G(\Phi,R)$ generated by the set $\{A\varepsilon_{\phi}(x)A^{-1}\mid x\in I,A\in G(\Phi,R)\}=:Z(I,\Phi)$ is called the \emph{elementary $I$-congruence subgroup of $G(\Phi,R).$}
\end{definition}

\begin{remark}
It is more common to consider $\bar{E}(\Phi,I)$ to be a normal subgroup of the so-called elementary Chevalley group $E(\Phi,R)$ generated by all root elements of $G(\Phi,R)$, but we are only interested in rings with $G(\Phi,R)=E(\Phi,R)$ in this paper.
\end{remark}

We mostly consider the two root systems $\Phi=A_n$ and $C_n.$ For these the corresponding Chevalley groups $G(A_n,R)$ and $G(C_n,R)$ are the two well-known matrix groups ${\rm SL}_{n+1}(R):=\{A\in R^{(n+1)\times (n+1)}\mid \det(A)=1\}$ and 
\[
{\rm Sp}_{2n}(R):=\{A\in R^{(2n)\times(2n)}\mid A^T\cdot J\cdot A=J\}
\text{ for } 
J:=
\begin{pmatrix}
0_n & I_n\\
-I_n & 0_n
\end{pmatrix}.
\]
With respect to these identifications (and a suitable choice of maximal torus) the root elements in $G(A_n,R)$ can be chosen as $E_{ij}(x):=I_{n+1}+xe_{ij}$ for $x\in R$ and $1\leq i\neq j\leq n+1$ with $e_{ij}$ the matrix whose sole non-zero entry is at the position $(i,j).$ 

In contrast the root elements in ${\rm Sp}_{2n}(R)$ for a suitable choice of maximal torus are slightly more complicated: The root elements associated to positive, short roots have the form $I_n+x(e_{ij}-e_{n+j,n+i}),I_n+x(e_{k,n+l}+e_{l,n+k})$ for $n\geq j>i\geq 1,1\leq l<k\leq n,k\neq l$ and $x\in R.$ The positive long root elements have the form $I_n+xe_{k,n+k}$ for $1\leq k\leq n$ and $x\in R.$ Additionally in ${\rm Sp}_{2n}(R)$, one has further negative root elements given by the transpose matrices of these positive root elements.

For $A=(a_{ij})\in{\rm GL}_n(R)$, one can further define the \emph{level ideal $l(A)\unlhd R$ of $A$} as the ideal in $R$ generated by the elements $a_{ij}$ and $a_{ii}-a_{jj}$ possibly with some restrictions on $i,j$ depending on the representation of $G(\Phi,\mathbb{C})$ chosen for the construction of $G(\Phi,\cdot)$. These elements are called \emph{level generators of }$A.$ This definition of $l(A)$ is obviously sensible for elements of $G(A_n,R)$ and $G(C_n,R)$ given the explicit description as matrix groups above, but as seen in \cite[Section~2]{General_strong_bound}, this definition makes sense also for general $G(\Phi,R)$ modulo a suitable representation of the corresponding $G(\Phi,\mathbb{C})$. We further define the ideal $l(T)\unlhd R$ for $T\subset G(\Phi,R)$ as the sum of the ideals $l(A)$ for $A\in T.$ Further, we recall the notation 
\[
\Pi(T):=\{\C P\text{ maximal ideal in }R\mid T\text{ maps into }Z\left(G(\Phi,R/\C P)\right)\}
\]
for $T\subset G(\Phi,R)$ from \cite{General_strong_bound}. Especially in Section~\ref{strong_bound_section}, we will often consider the same set $T$ as being contained in different $G(\Phi,R)$ for varying rings $R.$ To avoid confusion, we will consequently often write $\Pi_R(T)$ and $l_R(T)$ to make clear which $G(\Phi,R)$ is considered at any given point. 

\subsection{Conjugation-invariant norms and generalized conjugacy diameters}

The main object of study in this preprint are generalized conjugacy diameters on Chevalley groups:

\begin{definition}\cite[P.~1]{KLM}
Let $G$ be a group and $T$ a collection of conjugacy classes in $G$ generating $G.$ Then for $g\in G-\{1_G\}$, the number $\|g\|_T$ denotes the length of the shortest word in elements of $T$ and $T^{-1}$ needed to represent $g.$ (We further set $\|1_G\|_T:=0.$) Further, 
\begin{enumerate}
\item the number $\|G\|_T:=\sup\{\|g\|_T\mid g\in G\}\in\mathbb{N}_0\cup\{+\infty\}$ is called the \emph{diameter of $G$ wrt $\|\cdot\|_T$}. 
\item for $k\in\mathbb{N},$ the set $B_T(k):=\{g\in G\mid \|g\|_T\leq k\}\cup\{1_G\}$ is called the $k$-ball wrt $\|\cdot\|_T.$
\end{enumerate}
For $k\in\mathbb{N}$, the \emph{$k$.th generalized conjugacy diameter $\Delta_k(G)$ of $G$} is defined by 
\[
\Delta_k(G):=\sup\{\|G\|_T\mid T\text{ a collection of at most }k\text{ conjugacy classes in }G\text{ generating }G\}\in\mathbb{N}\cup\{+\infty\}
\]
\end{definition}

We further note that if $T$ consists of only one conjugacy class $B^G$ for $B\in G$, than we will write $\|\cdot\|_B$ and $B_B(k)$ instead of $\|\cdot\|_{\{B^G\}}$ or $B_{\{B^G\}}(k)$ for simplicity. Next, we introduce the following notation from \cite{KLM,General_strong_bound} that is useful when discussing conjugation-invariant word norms on Chevalley groups:

\begin{definition}
Let $R$ be a commutative ring with $1$, $\Phi$ an irreducible root system, $\phi\in\Phi, \phi_s\in\Phi$ short, $\phi_l\in\Phi$ long, $T$ a collection of conjugacy classes in $G(\Phi,R)$ and $k\in\mathbb{N}.$ We then set 
\begin{align*}
&\varepsilon_{\phi}(T,k):=\{x\in R\mid\varepsilon_{\phi}(x)\in B_T(k)\},\varepsilon_s(T,k):=\varepsilon_{\phi_s}(T,k),\varepsilon_l(T,k):=\varepsilon_{\phi_l}(T,k)\text{ and }\\
&\varepsilon(T,k):=\varepsilon_s(T,k)\cap\varepsilon_l(T,k).
\end{align*}
\end{definition}

\begin{remark}
There is the obvious question why the sets $\varepsilon_s(T,k),\varepsilon_l(T,k)$ and $\varepsilon(T,k)$ are well-defined. This can be  seen by considering the action of the Weyl group $W(\Phi)$ on root elements.
\end{remark}

Similar to the notation of $\Pi_R(T)$ and $l_R(T),$ we also use the clarifying notation of $B_T^{(R)}(k)$ for the $k$-ball $B_T(k)\subset G(\Phi,R)$ when there are multiple different relevant rings $R$ at play. 

\subsection{Rings of integers and stable range conditions}

Next, we recall the definition of global and local fields and their rings of integers. 

\begin{definition}
Let $K$ be a field.
\begin{enumerate}
\item A surjective homomorphism $v:(K^*,\cdot)\to\mathbb{Z}$ is called a \textit{non-archimedean valuation of $K$}, if it satisfies $v(a+b)\geq\min\{v(a),v(b)\}$ for all $a,b\in K^*$.(One further sets $v(0):=+\infty$ in this context.) The set $\mathscr{O}_v:=\{a\in K\mid v(a)\geq 0\}$ is called the \emph{ring of integers of $v$}, $m_v:=\{a\in K\mid v(a)>0\}$ \emph{its maximal ideal} and $\kappa_v:=\mathscr{O}_v/m_v$ its \emph{residue field}.
\item $K$ is called a \emph{local field}, iff it is $\mathbb{R}$ or $\mathbb{C}$ or if it admits a non-archimedean valuation $v$ such that $\kappa_v$ is finite and $K$ is complete wrt. to the topology induced by $v$ on it. In this case, the ring of integers of the local field $K$ is $\mathscr{O}_v.$
\item A \textit{real archimedean valuation of $K$} is a field homomorphism $K\to\mathbb{R}$ and a \emph{complex archimedean valuation of $K$} is a complex-conjugate pair of field homomorphisms $\phi,\bar{\phi}:K\to\mathbb{C}$ with $\phi(K)$ not contained in $\mathbb{R}.$ The set of all these archimedean valuations is denoted by $S_{\infty}.$
\item $K$ is called a \emph{global field} if it is a finite field extension of either $\mathbb{Q}$ or of the field of rational function $\mathbb{F}(T)$ for $\mathbb{F}$ a finite field. Finite extensions of $\mathbb{Q}$ are called \emph{number fields} and finite extensions of $\mathbb{F}(T)$ are called \emph{global function fields}.
\item If $K$ is a number field, then its \emph{ring of ($S_{\infty}$-)algebraic integers $\mathscr{O}_K:=\mathscr{O}_K^{S_{\infty}}$} is the integral closure of $\mathbb{Z}$ in $K.$ 
\item If $K$ is a number field not admitting an embedding $K\to\mathbb{R},$ then $K$ is called a \emph{totally imaginary number field.}
\item If $K$ is a global field and $S$ a non-empty, finite set of valuations of $K$ properly containing $S_{\infty}$, then \emph{the ring $\mathscr{O}_K^S$ of $S$-algebraic integers of $K$} is defined as 
\[
\mathscr{O}_K^S:=\{a\in K\mid\forall v\text{ valuation of }K\text{ with }v\notin S:v(a)\geq 0\}.
\]
\end{enumerate}
\end{definition}

We also want to note that each number field has at least one archimedean valuation (because it is algebraic over $\mathbb{Q}$ and so contained in the algebraically closed field $\mathbb{C}$). Further, for any finite, non-empty set of valuations $S$ of a global function field $K$, there is also a (non-unique) element $x_S\in K$ such that $S=\{v\text{ valuation of }K\mid v(x_S)<0\}$ and for all such $x_S$, the ring $\mathscr{O}_K^S$ is precisely the integral closure of the ring $\mathbb{F}[x_S]$ in $K.$ (A proof of this can be found in \cite[Theorem~14.5]{MR1876657}.)
We will also need various stable range concepts of rings a couple of times in the paper:

\begin{definition}\label{stability_conditions}
Let $R$ be a commutative ring with $1.$ Then for $k\in\mathbb{N}$, the ring $R$
\begin{enumerate}
\item has (Bass) stable range at most $k,$ iff for all $m\geq k$ and all $a_1,\dots,a_m,a_{m+1}\in R$ with $(a_1,\dots,a_m,a_{m+1})=R$, there are $x_1,\dots,x_m\in R$ such that $(a_1+x_1a_{m+1},a_2+x_2a_{m+1},\dots,a_m+x_ma_{m+1})=R.$ 
\item satisfies the absolute stable range condition $ASR(k),$ iff for all $a_1,\dots,a_k\in R$, there are $x_1,\dots,x_{k-1}$ such that each maximal ideal $\C P$ of R containing the ideal $(a_1+x_1a_k,a_2+x_2a_k\dots,a_{k-1}+x_{k-1}a_k)$ also contains $(a_1,\dots,a_k).$
\item has ideal stable range at most $k,$ iff for all $m\geq k$ and each ideal $I$ of $R$ and elements $v_1,\dots,v_m,v_{m+1}$ generating $I$ as an ideal, there are $x_1,\dots,x_m\in R$ such that $v_1+x_1v_{m+1},\dots,v_m+x_mv_{m+1}$ are also generating $I$ as an ideal.
\end{enumerate}
\end{definition}

These concepts are quite classical. For example, Bass stable range was introduced by Bass in \cite{MR0174604} (albeit indexed differently), absolute stable range by Stein \cite{197877} and ideal stable range was studied among else in \cite{stepanov1989stable}. 

\section{Bounded generation by conjugates of root elements in Chevalley groups}

The main goal of this section is the following variant of Theorem~\ref{main_thm2} used to prove Theorem~\ref{main_thm0} in Section~\ref{strong_bound_section}:

\begin{theorem}\label{main_thm2_technical}
Let $K$ be a global field and $S$ a finite non-empty set of valuations of $K$ containing all archimedean valuations of $K$. Let $R:=\mathscr{O}_K^S$ be the ring of all S-algebraic integers of $K.$ Further, let $\Phi\in\{A_{n\geq 2},C_{n\geq 2},E_6,E_7,E_8,F_4,G_2\}$ be given and let $I$ be a non-zero ideal in $R$. Assume further, if $K$ is a number field that either $K$ has a real embedding or $I$ is coprime to the number $m(K)$ of roots of unity of $K$, when considering $m(K)$ as an element of $R.$ If $R$ is a principal ideal domain or $K$ is a global function field, set $\Delta:=\Delta(K):=3.$ If $K$ is not a global function field and $R$ not a principal ideal domain, let $\Delta:=\Delta(K)$ be the maximum of the number of rational prime divisors of the discriminant of $K$ and the number $3$. Further, set $L_A:=L_A(K):=68\Delta(K)+11$ and $L_C:=L_C(K):=180\Delta(K)+26.$
Then
\begin{equation*}
\|\bar{E}(\Phi,I)\|_{Z(I,\Phi)}\leq
\begin{cases}
L_A+8(n-2)&\text{, if }\Phi=A_n,\\
L_C+15(n-2)&\text{, if }\Phi=C_n,\\
L_A|\Phi^+|&\text{, if }\Phi=E_6,E_7,E_8,\\
24L_C&\text{, if }\Phi=F_4,\\
L_A+21&\text{, if }\Phi=G_2
\end{cases}
\end{equation*}
holds.
\end{theorem}

\begin{remark}
The $A_2$-case for number fields is almost directly taken from \cite[Corollary~3.3]{gvozdevsky2023width} and the result for the other roots system also uses the strategy from \cite{gvozdevsky2023width}. However, one could avoid \cite{gvozdevsky2023width} entirely and replace it with results from \cite{MR2357719,Chevalley_positive_char_tentative}. This approach could still be used to prove Theorem~\ref{main_thm0}, but the value of $D(K)$ would consequently be completely unclear beyond the vague notion that it has to be chosen larger than some threshold determined by an unknown function admitting only the numbers $[K:\mathbb{Q}]$ (in the number field case) and $[K:\mathbb{F}(T)],|\mathbb{F}|$ (in the global function field case) as arguments. Using the approach from \cite{gvozdevsky2023width} produces explicit thresholds (at least in most cases) and makes it clear which quantities have to be determined to make non-explicit threshold functions explicit. 
\end{remark}

Before coming to proof of Theorem~\ref{main_thm2_technical}, we want to note that Theorem~\ref{main_thm2} is a consequence of it. For this purpose, we recall the setup of Theorem~\ref{main_thm2}: Let $F$ be an algebraic extension of the finite field $\mathbb{F}$, $C$ an irreducible, nonsingular, geometrically integral, projective curve defined over $F$, $U$ a proper, non-empty Zariski-open set in $C$ defined over $F$. Further, let $R:=F_C[U]=\mathscr{O}_C(U)$ be the ring of $F$-defined, regular functions on $U$ and $\Phi\in\{A_{n\geq 2},C_{n\geq 2},E_6,E_7,E_8,F_4,G_2\}.$ Let $X\in C(\Phi,I)$ be given; then there is a finite field extension $\mathbb{E}$ of $\mathbb{F}$ contained in $F$ such that the following two conditions hold: 
\begin{enumerate}
\item The curve $C$ and the Zariski-open set $U$ are already defined over $\mathbb{E}$ and
\item considering the subring $S:=\mathbb{E}_C[U]$ of $\mathbb{E}$-defined regular functions and the ideal $J:=I\cap S$, one has that $X$ is an element of the subgroup $C(\Phi,J)$ of $G(\Phi,S).$
\end{enumerate}
Thus if we can prove Theorem~\ref{main_thm2} for the ring $S$ and its ideal $J$, we obtain Theorem~\ref{main_thm2_technical} in general (given that $X$ was arbitrary and the bounds in Theorem~\ref{main_thm2_technical} are independent of the global function field and ideal in question) and so Theorem~\ref{main_thm2} indeed follows from Theorem~\ref{main_thm2_technical}. The main definition needed to prove Theorem~\ref{main_thm2_technical} is the double of a ring:

\begin{definition}
Let $R$ be a ring and $I$ a proper ideal in $R.$ Then define $\tilde{R}_I$ (or $\tilde{R}$, if there is no risk of confusion) as the subring of the cartesian product ring $R^2$ given by the following subset $\{(a,b)\in R^2\mid a\equiv b\text{ mod }I\}.$
\end{definition}

\begin{remark}
These rings were first introduced in Stein's paper \cite{STEIN1971140} in which he studies the stability properties of these type of rings. We will quote and use some of these results in Subsubsection~\ref{bounded_gen_g2} below.
\end{remark}

For a commutative ring $R$ with $1$ and an irreducible root system, the number $\|E(\Phi,R)\|_E$ denotes the diameter of $E(\Phi,R)$ wrt the (\emph{non-conjugation-invariant}) word norm given by its generating set of root elements. One proves Theorem~\ref{main_thm2_technical} by connecting it to bounded elementary generation of $G(\Phi,\tilde{R})$ via the following lemma:

\begin{lemma}\cite[Proof of Corollary~3.3]{gvozdevsky2023width}\label{bounded_gen_double_implies_conj_gen}
Let $R$ be a commutative ring with $1$ and $I$ a proper ideal in $R$, $\Phi$ an irreducible root system and $K\in\mathbb{N}.$ Further, assume that $E(\Phi,R)=G(\Phi,R).$ Then $\|E(\tilde{R},\Phi)\|_E\leq K$ implies that $\|\bar{E}(\Phi,I)\|_{Z(I,\Phi)}\leq K.$
\end{lemma}

The paper \cite{gvozdevsky2023width} only contains a special version of this fact for $\Phi=A_2$, but the proof works in general, once one has checked that $\tilde{A}:=(A,1)$ is an element of $E(\Phi,\tilde{R})$ for any $A\in\bar{E}(\Phi,I)$ and $1$ the neutral element of $G(\Phi,R)$. We leave checking this as an exercise for the reader. Lemma~\ref{bounded_gen_double_implies_conj_gen} allows to deduce from bounded generation by root elements for Chevalley groups over $\tilde{R}$, bounded generation by conjugates of root elements for the corresponding congruence subgroup $\bar{E}(\Phi,I).$ However, the method described by Lemma~\ref{bounded_gen_double_implies_conj_gen} does not behave well with respect to conjugates of root elements in $G(\Phi,\tilde{R})$ and as a consequence does not produce the required asymptotics of $\|\bar{E}(\Phi,I)\|_{Z(I,\Phi)}$ to prove Theorem~\ref{main_thm2_technical}. Thus the proof of Theorem~\ref{main_thm2_technical} has two main parts for $\Phi\in\{A_{n\geq 2},C_{n\geq 2}\}$: First, the bounded generation result for $G(A_2,\tilde{R})$ and $G(C_2,\tilde{R})$ itself:

\begin{proposition}\label{bounded_gen_tilde}
Let $K$ be a global field and $S$ a non-empty set of valuations of $K$ containing all archimedean valuations of $K$. Then consider the ring of S-algebraic integers $R:=\mathscr{O}_K^S$. Assume further, if $K$ is a number field that $K$ has a real embedding or $I$ is coprime to the number $m(K)$ of roots of unity of $K$, when considering $m(K)$ as an element of $R.$ Let $I$ be an ideal in $R$ and $\tilde{R}$ the corresponding double and set $\Delta=\Delta(K)$ as in Theorem~\ref{main_thm2_technical}. Then each element of the group $G(A_2,\tilde{R})$ is a product of $68\Delta+11$ many root elements and each element of $G(C_2,\tilde{R})$ is a product of $180\Delta+26$
many root elements.
\end{proposition}

Second and more interestingly, one needs an effective K-theoretic stability argument contained in the following proposition

\begin{proposition}\label{stability}
Let $X\in\{A,C\}$ and $n\geq 2$ be given. Further, let $R$ be a Dedekind domain and $I$ an ideal in $R$. Then each element of $C(X_n,I)$ is a product of $8(n-2)$ in the case of $X=A$ and $15(n-2)$ in the case of $X=C$ many elements of ${Z(I,X_n)}$ and an element of $C(X_2,I).$
\end{proposition}

\begin{remark}
\hfill
\begin{enumerate}
\item
Stability results of a similar type are quite well-known and can be found in \cite{ MR0174604,sinchuk2018decompositions,197877,stepanov1989stable} for example. However, due to the non-explicit nature of the corresponding proofs, the number of elements from $Z(I,X_n)$ needed remains unclear in these papers: A straight reading of the proof of \cite[Theorem~4.2]{197877} for example would imply a quadratic (in $n$) number of elements of $Z(I,X_n)$ rather than the linear bound required in this paper. In any case, the stability property used in the proof of Proposition~\ref{stability} is the ideal stable range of a ring from Definition~\ref{stability_conditions}(3).
\item We restrict ourselves to the root systems of $\Phi=A_n$ and $C_n.$ It seems reasonable that the other root systems $B_n$ and $D_n$ would work as well, but they would require further matrix calculations. 
\end{enumerate}
\end{remark}

Note in particular that the exceptional root systems $\Phi=E_6,E_7,E_8$ and $F_4$ don't have to be considered for the purposes of Theorem~\ref{main_thm2_technical} after proving Proposition~\ref{bounded_gen_tilde}: Using the proof of \cite[Proposition~1]{MR1044049}, given an irreducible root system $\Phi$, a commutative ring $S$ and a natural number $K$, such that one has $\|E(\Psi,S)\|_E\leq K$ for each irreducible root subsystem $\Psi$ of rank $2$ of $\Phi$, then $\|E(\Phi,S)\|_E\leq K|\Phi^+|$ holds. Thus one obtains Theorem~\ref{main_thm2_technical} for the higher rank root systems $E_6,E_7,E_8$ and $F_4$ from Proposition~\ref{bounded_gen_tilde} and Lemma~\ref{bounded_gen_double_implies_conj_gen} directly. The case of $G_2$ for Theorem~\ref{main_thm2_technical} requires an additional argument, which we will supply below. 

\subsection{Bounded generation}

\subsubsection{Bounded generation for $A_n$}\label{bounded_generation_subsection_A2}

\begin{lemma}\label{tilde_bounded_generation_A2}
Let $K$ be a global function field and $T$ a finite non-empty set of valuations of $K$. Let $R:=\mathscr{O}_K^T$ be the ring of all T-algebraic integers of $K$ and $I$ a non-zero ideal in $R.$ Then for $\tilde{R}$ the ring associated to $R$ and $I$ each element of ${\rm SL}_3(\tilde{R})$ is a product of at most $215$ elementary matrices. 
\end{lemma}

Before coming to the proof, we need to introduce a piece of notation: For a ring $S$, the quantity $\epsilon(S)\in\mathbb{N}\cup\{+\infty\}$ denotes the exponent of the group of units $S^*$. Further, for a ring $R$ and an element $x\in R$, the quantity $\epsilon(x)\in\mathbb{N}\cup\{+\infty\}$ denotes the exponent of the group of units $(R/xR)^*$ for the quotient ring $R/xR.$ Additionally, we need the following lemma which is a combination of \cite[Lemma~4.3]{gvozdevsky2023width} and the proof of \cite[Lemma~3.3]{https://doi.org/10.1112/blms.12925}: 

\begin{lemma}\label{gcd_lemma_double}
Let $K$ be a global function field with field of constants $\mathbb{F}$ and $T$ a finite non-empty set of valuations of $K$. Let $R:=\mathscr{O}_K^T$ be the ring of all T-algebraic integers of $K$, $\mathfrak{a}$ a non-zero ideal in $R, m:=|\mathbb{F}^*|$ and $b\in R$ given such that $bR$ is a prime ideal and coprime to $\mathfrak{a}.$ Further, set $e_p:=v_p(m)$ for each rational prime and let $\mathbb{S}$ be a set of rational primes such that one of the following conditions holds:
\begin{enumerate}
\item For all $p\in\mathbb{S}$, one has $v_p(\epsilon(b))>e_p$. 
\item For all $p\in\mathbb{S}$, one has $v_p(\epsilon(b))=e_p$. 
\end{enumerate}
Then there is an element $c\in R$ and a $\gamma\in\mathbb{N}$ not divisible as a natural number by any element of $\mathbb{S}$ such that $bc\equiv -1\text{ mod }\mathfrak{a}$ and $\epsilon(c)=m\gamma.$
\end{lemma}

\begin{proof}
We distinguish two cases depending on the property satisfied by $\mathbb{S}$. The first case is almost identical to the proof of \cite[Lemma~3.3]{https://doi.org/10.1112/blms.12925}: Note first that the set $S$ in the proof of \cite[Lemma~3.3]{https://doi.org/10.1112/blms.12925} contains $\mathbb{S}$ by its very definition. Rather than repeating the proof verbatim, we note that at the end of the proof of \cite[Lemma~3.3]{https://doi.org/10.1112/blms.12925}, one shows for the $c$ produced there that $bc\equiv -1\text{ mod }\mathfrak{a}$ and that $v_p(\epsilon(c))=e_p=v_p(m)$ holds for each 
$p\in S$. Additionally, $c$ satisfies $cR=\C P_1\cdot\C P_2$ for two prime ideals $\C P_1$ and $\C P_2$ in $R$ (which can be assumed to be different) and so $\epsilon(c)$ is a multiple of $|(R/\C P_1)^*|$ and so a multiple of $m.$ Hence $\epsilon(c)=m\cdot\gamma$ for some natural number $\gamma.$ But $\gamma$ has no prime factor $p\in\mathbb{S}$, because $v_p(\epsilon(c))=v_p(m)$ holds for all $p\in S$. This finishes the first case.

For the second case, we pick some algebraic closure $F$ of the finite field $\mathbb{E}:=R/bR$. Then for each $p\in\mathbb{S}$, we pick a primitive $p^{e_p+1}$.th root of unity $\xi_p$ in $F.$ Then let $K_{\mathfrak{a}}$ be the ray class field associated to the ideal $\mathfrak{a}$ and set $L_{\mathfrak{a}}:=K_{\mathfrak{a}}\cdot\mathbb{F}[\xi_p\mid p\in\mathbb{S}].$ As in the proof of \cite[Lemma~3.3]{https://doi.org/10.1112/blms.12925}, one sees that that the ideal $bR$ is unramified in $L_{\mathfrak{a}}|K$ and so $\sigma:=(bR,L_{\mathfrak{a}}|K)$ is a well-defined element of ${\rm Gal}(L_{\mathfrak{a}}|K).$ Then according to Tchebotarevs density theorem \cite[Theorem~9.13A]{MR1876657}, we may choose a prime $\C P$ unramified in $L_{\mathfrak{a}}|K$ such that $(\C P,L_{\mathfrak{a}}|K)=\sigma^{-1}$ holds. Then as in the proofs of \cite[Lemma~4.3]{gvozdevsky2023width} and \cite[Lemma~3.3]{https://doi.org/10.1112/blms.12925}, one can find a $c\in R$ such that $cR=\C P,bc\equiv -1\text{ mod }\mathfrak{a}$ and $(cR,L_{\mathfrak{a}}|K)=\sigma^{-1}$. Next, we assume for contradiction that there is a $p\in\mathbb{S}$ such that $v_p(\epsilon(c))>e_p.$ But this implies that $R/\C P=R/cR$ must contain all primitive roots $p^{e_p+1}$.th root of unity and in particular $\xi_p$.

Next, let $Q_1$ be a prime in $\C O_{L_p}$ covering $\C P$ for $\C O_{L_p}$ the integral closure of $R$ in $L_p:=\mathbb{F}[\xi_p]\cdot K.$ Then $\C O_{L_p}/Q_1=(R/\C P)[\xi_p]$ holds according to \cite[Proposition~8.10]{MR1876657} and so $\C O_{L_p}/Q_1=R/\C P$. Further, by \cite[Proposition~9.10]{MR1876657} the order of $\sigma|_{L_p}=(\C P,L_p|K)^{-1}$ agrees with the degree $[\C O_{L_p}/Q_1:(R/\C P)]=1.$ But then $\sigma|_{L_p}=(bR,L_p|K)$ is trivial. Hence picking a prime ideal $\C Q_2$ in $\C O_{L_P}$ lying over $bR$, one sees by \cite[Proposition~9.10]{MR1876657} that $\C O_{L_p}/Q_2=R/bR$. But $\C O_{L_p}/Q_2=(R/bR)[\xi_p]$ holds according to \cite[Proposition~8.10]{MR1876657}. Hence $\xi_p\in R/bR.$ So in particular $p^{e_p+1}={\rm ord}(\xi_p)$ would divide $\epsilon(b).$ Hence $v_p(\epsilon(b))>e_p$, a contradiction. 

So $v_p(\epsilon(c))=e_p$ holds for all $p\in\mathbb{S}.$ Further, $cR=\C P$ is a prime ideal in $R$ and so as before $m$ divides $\epsilon(c).$ In summary, $\epsilon(c)=m\gamma$ holds with $\gamma$ not divisible by any element from $\mathbb{S}.$ This finishes the second case.
\end{proof}

Now, we can finally prove Lemma~\ref{tilde_bounded_generation_A2}:

\begin{proof}
Citing results from \cite{ESTES1967343} and \cite{STEIN1971140}, one can show similar to arguments in Subsection~\ref{bounded_gen_g2}, that the ring $\tilde{R}$ also has stable range at most $2$ and so using the standard stable range arguments (as for example in the proof of \cite[Main Theorem]{MR704220}), one can reduce any element of ${\rm SL}_3(\tilde{R})$ to an element of its subgroup ${\rm SL}_2(\tilde{R})$ by multiplying with $7$ elementary matrices. The proof then proceeds by applying \cite[Proposition~4.1]{gvozdevsky2023width} and is generally very similar to the one in \cite{gvozdevsky2023width}. To apply \cite[Proposition~4.1]{gvozdevsky2023width} we need to show two claims for $m:=|\mathbb{F}^*|$:
\begin{claim}\label{Claim1_bounded_gen_sl}
For any $v=(a_1,b_1)\in\tilde{R}^2$ with $a_1\tilde{R}+b_1\tilde{R}=\tilde{R}$, there is a product $X$ of $4$ elementary matrices in ${\rm SL}_2(\tilde{R})$ such that $v\cdot X=(a^m,b)$ for $a=(a',a''),b=(b',b'')\in\tilde{R}$ such that $a',a'',b',b''$ are all prime to $I$ and $b'R,b''R$ are prime ideals. (For the sake of brevity, we will call vectors $(a,b)\in\tilde{R}$ with $a',a'',b',b''$ all prime to $I$, $b'R,b''R$ prime ideals and $a\tilde{R}+b\tilde{R}=\tilde{R}$ \emph{post-processed} later in the proof.)
\end{claim}
We first note that given that $J:=I\times I$ is an ideal in $\tilde{R}$, and clearly $\tilde{R}/J$ is a subring of the finite ring $(R/I)^2$, that the ring $\tilde{R}/J$ is semi-local. Hence we can find an element $x\in\tilde{R}$ such that 
$b_2:=xa_1+b_1$ maps to a unit in $\tilde{R}/J.$ But this implies that for $b_2=(b_2',b_2'')$ both $b_2'$ and $b_2''$ are coprime to $I.$ So the elementary transvection $E_{12}(x)$ transforms $(a_1,b_1)$ into $(a_1,b_2).$ Next, note that for $a_1=(a_1',a_1''),$ we have that $a_1'$ and $b_2'$ are coprime. The next part of the proof is largely taken from our earlier paper \cite{https://doi.org/10.1112/blms.12925}. We recall that the global function field $K$ is a finite field extension of a rational function field $\mathbb{F}(x)$ for a function $x\in K$ such that $R$ is the integral closure of the ring $\mathbb{F}[x].$ Let $S=\{R_{\infty}^{(1)},\dots,R_{\infty}^{(L)}\}$ further be the set of primes covering the prime $\C P_{\infty}$ at infinity of $\mathbb{F}(x)$ (aka the poles of $x\in K$). Then according to \cite[A16]{MR244257}, for the function
\begin{equation*}
\left(\frac{\cdot,\cdot}{R_{\infty}^{(1)}}\right)_m:K_{R_{\infty}^{(1)}}^*\times K_{R_{\infty}^{(1)}}^*\to\mathbb{F}^*
\end{equation*}
there are $u,w\in K_{R_{\infty}^{(1)}}^*$ such that 
\begin{equation*}
\left(\frac{u,w}{R_{\infty}^{(1)}}\right)_m
\end{equation*}
generates $\mathbb{F}^*.$ Using Dirichlet's Theorem (and the fact that $b_2'$ and $I$ are coprime), we can next choose a principal prime ideal $a_2'R$ satisfying the following properties for all $i>1$: 
\begin{align*}
a_2'\equiv a_1'\text{ mod }b_2'R,a_2'\equiv 1\text{ mod }I,\forall i=2,\dots,L:a_2'\equiv 1\text{ mod }R_{\infty}^{(i)}\text{ and }a_2'\equiv w\text{ mod }(R_{\infty}^{(1)})^N
\end{align*}
with $N$ chosen so large that $a_2'/w$ has an m.th root in $K_{R_{\infty}^{(1)}}.$ Similarly, one can choose a prime ideal $a_2''R$ such that for all $i>1:$
\begin{align*}
a_2''\equiv a_1''\text{ mod }b_2''R,a_2''\equiv 1\text{ mod }I,\forall i=2,\dots,L:a_2''\equiv 1\text{ mod }R_{\infty}^{(i)}\text{ and }a_2''\equiv w\text{ mod }(R_{\infty}^{(1)})^N
\end{align*}
with $N$ chosen so large that $a_2''/w$ has an m.th root in $K_{R_{\infty}^{(1)}}.$ Given that $a_2'-a_2''=(a_2'-1)+(1-a_2'')\in I$ one sees $a_2:=(a_2',a_2'')\in\tilde{R}$. But we can find $y',y''\in R$ such that  
\[
a_2'-a_1'=y'b_2'\text{ and }a_2''-a_1''=y''b_2''.
\]
But now slightly abusing notation, we obtain 
\[
y'\equiv (a_2'-a_1')b_2'^{-1}\equiv (a_2''-a_1'')b_2''^{-1}\equiv y''\text{ mod }I
\]
and so $y:=(y',y'')$ is an element of $\tilde{R}.$
Then $(a_1,b_2)$ can be transformed into $(a_2,b_2)$ by multiplication with $E_{21}(y)$. Next, note 
\begin{align*}
\left(\frac{u,a_2'}{R_{\infty}^{(1)}}\right)_m=\left(\frac{u,a_2'/w}{R_{\infty}^{(1)}}\right)_m\cdot\left(\frac{u,w}{R_{\infty}^{(1)}}\right)_m=\left(\frac{u,w}{R_{\infty}^{(1)}}\right)_m.
\end{align*}
Hence there is a positive integer $k$ such that 
\begin{align*}
\left(\frac{b_2',a_2'}{R_{\infty}^{(1)}}\right)_m\cdot\left(\frac{u,a_2'}{R_{\infty}^{(1)}}\right)_m^k=1
\end{align*}
Set $v:=u^k.$ But $a_2'$ and $I$ are coprime and so using Dirichlet's Theorem again, we can find a prime ideal $b'R$ satisfying the following properties for all $i>1:$
\begin{align*}
b'\equiv b_2'\text{ mod }a_2'R,b'\equiv 1\text{ mod }I,\forall i=2,\dots,L:b'\equiv 1\text{ mod }R_{\infty}^{(i)}\text{ and }b'\equiv v\text{ mod }(R_{\infty}^{(1)})^N
\end{align*}
Again, the integer $N$ is chosen so large that $b'/v$ has an m.th root in $K_{R_{\infty}^{(1)}}.$ Then according to the m.th power reciprocity law, one has
\begin{equation*}
1=\prod_{\C P} \left(\frac{b',a_2'}{\C P}\right)_m
\end{equation*}
with the product taken over all primes in $K.$ But note that due to the choice of $a_2'R,b'R$ as prime ideals and the fact that $a_2'$ and $b'$ are chosen to be congruent to $1$ modulo the primes $R_{\infty}^{(i)}$ for $i\geq 2,$ this product reduces to 
\begin{equation}\label{claim1_equation}
1=\left(\frac{b',a_2'}{b'\C R}\right)_m\cdot\left(\frac{b',a_2'}{a_2'R}\right)_m\cdot\left(\frac{b',a_2'}{R_{\infty}^{(1)}}\right)_m
\end{equation} 
according to \cite[A16]{MR244257}. Note that $b'/v$ has an m.th root in $K_{R_{\infty}^{(1)}}$ and so
\begin{equation*}
\left(\frac{b',a_2'}{a_2'R}\right)_m\cdot\left(\frac{b',a_2'}{R_{\infty}^{(1)}}\right)_m=
\left(\frac{b',a_2'}{a_2'R}\right)_m\cdot\left(\frac{v,a_2'}{R_{\infty}^{(1)}}\right)_m=
\left(\frac{b',a_2'}{a_2'R}\right)_m\cdot\left(\frac{u^k,a_2'}{R_{\infty}^{(1)}}\right)_m=
\left(\frac{b',a_2'}{a_2'R}\right)_m\cdot\left(\frac{u,a_2'}{R_{\infty}^{(1)}}\right)_m^k=1
\end{equation*} 
holds. This implies together with (\ref{claim1_equation}) that
\begin{equation*}
\left(\frac{a_2'}{b'R}\right)_m^{-1}=\left(\frac{a_2'}{b'R}\right)_m^{-{\rm ord}_{b'R}(b'R)}=\left(\frac{b',a_2'}{b'R}\right)_m=1.
\end{equation*}
Hence $a_2'$ is an m.th power modulo $b'R$. The same way, one can find a prime ideal $b''R$ such that 
\[
b''\equiv b_2''\text{ mod }a_2''R,b''\equiv 1\text{ mod }I,a_2''\text{ is an m.th power modulo }b''R.
\]
But note that $b:=(b',b'')\in\tilde{R}$ and so, as before by multyplying with $E_{12}(s)$ for a suitable $s\in\tilde{R}$, we can pass from $(a_2,b_2)$ to $(a_2,b).$ Further, by \cite[Lemma~2.1]{gvozdevsky2023width}, there is an isomorphism of rings
\[
\tilde{R}/b\tilde{R}\to (R/b'R)\times(R/b''R)
\]
given by the projection maps. Hence $a_2$ maps onto an m.th power in $\tilde{R}/b\tilde{R}$. As $I$ is coprime to both $b'$ and $b''$, there are $a=(a',a''),t\in\tilde{R}$ with $a_2-bt=a^m$ and $a',a''$ coprime to $I.$ Thus after multiplying with $E_{21}(t)$ we are done with the proof of Claim~\ref{Claim1_bounded_gen_sl}. Next, we show:

\begin{claim}
For any post-processed $(a,b)\in\tilde{R}^2$, there are $c_1,c_2,c_3\in\tilde{R}$ such that $bc_i\equiv -1\text{ mod }a\tilde{R}$ and $\gcd(\epsilon(b),\epsilon(c_i))=m\delta_i$ holds for $i=1,2,3$ with ${\rm gcd}(\delta_1,\delta_2,\delta_3)=1.$
\end{claim}

For $b=(b',b'')$ according to \cite[Lemma~2.1]{gvozdevsky2023width}, we have $\tilde{R}/b\tilde{R}=(R/b'R)\times(R/b''R)$ and so $\epsilon(b)=\text{lcm}(\epsilon(b'),\epsilon(b'')).$ In particular, $m$ is a divisor of $\epsilon(b)$ and so for each rational prime number, one has $v_p(\epsilon(b))\geq v_p(m)=:e_p.$ We set
$$
\mathbb{S}:=\{p\text{ a rational prime divisor of }\epsilon(b)\mid v_p(\epsilon(b))>e_p\}
$$
and further set 
\begin{align*}
&\mathbb{S}_1:=\{p\in \mathbb{S}\mid v_p(\epsilon(b'))>e_p=v_p(\epsilon(b''))\},\\
&\mathbb{S}_2:=\{p\in \mathbb{S}\mid v_p(\epsilon(b'))>e_p<v_p(\epsilon(b''))\}\text{ and }\\
&\mathbb{S}_3:=\{p\in \mathbb{S}\mid v_p(\epsilon(b'))=e_p<v_p(\epsilon(b''))\}.
\end{align*}
As $\epsilon(b)=\text{lcm}(\epsilon(b'),\epsilon(b''))$ holds, we have that $\mathbb{S}$ is a disjoint union of $\mathbb{S}_1,\mathbb{S}_2$ and $\mathbb{S}_3.$ But now note that by applying Lemma~\ref{gcd_lemma_double} to the two setups of  $\mathfrak{a}=a'I,b'R,\mathbb{S}_1$ and $\mathfrak{a}=a''I,b''R,\mathbb{S}_1$ respectively, one finds two elements $c_1',c_1''\in R$ such that 
$$
b'c_1'\equiv -1\text{ mod }a'I,\epsilon(c_1')=m\gamma_1',b''c_1''\equiv -1\text{ mod }a''I,\epsilon(c_1'')=m\gamma_1''
$$
with $\gamma_1',\gamma_1''$ both not being divisible by any element of $\mathbb{S}_1.$ Further, note by slight abuse of notation that 
$$
c_1'\equiv (-b')^{-1}\equiv (-b'')^{-1}\equiv c_1''\text{ mod }I.
$$
Thus $c_1:=(c_1',c_1'')$ is an element of $\tilde{R}$ and as both $c_1',c_1''$ are coprime to $I$, one obtains by \cite[Lemma~2.1]{gvozdevsky2023width} that $\epsilon(c_1)={\rm lcm}(\epsilon(c'_1),\epsilon(c''_1))={\rm lcm}(m\gamma_1',m\gamma_1'')=m\cdot{\rm lcm}(\gamma_1',\gamma_1'').$ But setting $\gamma_1:={\rm lcm}(\gamma_1',\gamma_1'')$, one sees that $\gamma_1$ is not divisible by any element of $\mathbb{S}_1.$
Next, note that ${\rm gcd}(\epsilon(b),\epsilon(c_1))={\rm gcd}(\epsilon(b),m\gamma_1)$ is a multiple of $m,$ so it is $m\delta_1$ for some $\delta_1\in\mathbb{N}.$ Next, let $p$ be a rational prime divisor of $\delta_1$. Clearly, 
$$
e_p=v_p(m)<v_p(m)+v_p(\delta_1)=v_p(m\delta_1)\leq v_p(\epsilon(b))
$$
holds and so $p\in \mathbb{S}.$ So all prime divisors of $\delta_1$ are elements of $\mathbb{S}.$ If $p$ is a prime divisor of $\delta_1$ in $\mathbb{S}_1,$ then $p$ is not a divisor of $\gamma_1$. Hence $p$ divides $m$ and so
$$
e_p<v_p(m\delta_1)=v_p({\rm gcd}(\epsilon(b),m\gamma_1))={\rm min}(v_p(\epsilon(b)),v_p(m\gamma_1))=v_p(m)=e_p.
$$
This contradiction shows that $\delta_1$ has no prime divisors in $\mathbb{S}_1.$ So all prime divisors of $\delta_1$ are elements of $\mathbb{S}-\mathbb{S}_1.$ Further, as $a',a''$ are coprime to $I$, one has $bc_1\equiv -1\text{ mod }a\tilde{R}$ by \cite[Lemma~2.1]{gvozdevsky2023width}. The same way, one finds elements $c_2,c_3\in\tilde{R}$ such that 
$$
bc_2\equiv -1\text{ mod }a\tilde{R}\text{ and }{\rm gcd}(\epsilon(b),\epsilon(c_2))=m\delta_2,bc_3\equiv -1\text{ mod }a\tilde{R}\text{ and }{\rm gcd}(\epsilon(b),\epsilon(c_3))=m\delta_3
$$
holds with all prime divisors of $\delta_2$ in $\mathbb{S}-\mathbb{S}_2$ and all prime divisors of $\delta_3$ in $\mathbb{S}-\mathbb{S}_3$.
But given that $\mathbb{S}$ is a disjoint union of $\mathbb{S}_1,\mathbb{S}_2,\mathbb{S}_3$, one now obtains ${\rm gcd}(\delta_1,\delta_2,\delta_3)=1.$ But now one can apply \cite[Proposition~4.1]{gvozdevsky2023width} to conclude that each element of ${\rm SL}_3(\tilde{R})$ is a product of $7+4+68*3=215$ elementary matrices and this finishes the proof.
\end{proof}

Lemma~\ref{tilde_bounded_generation_A2} together with Lemma~\ref{bounded_gen_double_implies_conj_gen} implies the $A_2$-case for global function fields in Proposition~\ref{bounded_gen_tilde}. The number field-case is precisely the content of \cite[Theorem~3.1]{gvozdevsky2023width} in combination with the seven root elements from the stable range argument mentioned at the beginning of the proof of Lemma~\ref{tilde_bounded_generation_A2}.

\subsubsection{Bounded generation for $C_n$}

\begin{lemma}\label{tilde_bounded_generation_C2}
Let $K$ be a global field and $S$ a finite non-empty set of valuations of $K$ containing all archimedean valuations of $K$. Let $R:=\mathscr{O}_K^S$ be the ring of all S-algebraic integers of $K$ and $I$ a non-zero ideal in $R.$ Further, if $K$ is a number field assume also that either $I$ is coprime to the number of roots of unity in $K$ or $K$ has a real embedding. Set $\Delta=\Delta(K)$ further as in Theorem~\ref{main_thm2_technical}. Then for $\tilde{R}$ the ring associated to $R$ and $I$ each element of ${\rm Sp}_4(\tilde{R})$ is a product of at most $180\Delta+26$ many root elements.
\end{lemma}

\begin{proof}
The proof strategy is essentially an adaptation of the strategy from Gvozdevsky's \cite{gvozdevsky2023width} applied to the corresponding proof in \cite{MR1044049} showing bounded generation by root elements for ${\rm Sp}_4(R)$ where $R$ is a ring of S-algebraic integers in a number field. But rather than tediously going through the entire argument, as we did for ${\rm SL}_3$ to show that everything works the same way for $\tilde{R}$ as it does for $R$, we only touch on the main differences: The first part of the proof for ${\rm Sp}_4(\tilde{R})$ consists of reducing a given element $A\in{\rm Sp}_4(\tilde{R})$ to an element $A_2$ of the subgroup 
\[
{\rm Sp}_4^{\beta}(\tilde{R})=\left\{
\begin{pmatrix}
1 & 0 & 0 & 0\\
0 & a & 0 & b\\
0 & 0 & 1 & 0\\
0 & c & 0 & d
\end{pmatrix}\in\tilde{R}^{4\times 4}\mid ad-bc=1
\right\}
\]
of ${\rm Sp}_4(\tilde{R})$ by multiplying with $13=11+2$ elementary matrices. This step essentially follows by doing the stability argument of \cite[P.109]{MR1044049} twice (with the $+2$ to ensure that the second component of the $(1,1)$ and $(2,1)$-entries of $A$ are coprime to $I$ to begin with). Second, one has to show the following claim:

\begin{claim}\label{square_claim}
Each element $B\in{\rm SL}_2(\tilde{R})$ with first row $(u,v)$ can by multiplying $B$ with four elementary matrices be transformed into an element with first row $(u_1,v_1^2).$
\end{claim}

In the number field case this claim is simply a consequence of \cite[Lemma~4.2]{gvozdevsky2023width}. For global function fields, we have to distinguish the two cases of odd and even characteristics. The first case is shown in Claim~\ref{Claim1_bounded_gen_sl} as $|\mathbb{F}|-1$ is a multiple of $2$. For the even case note first that applying Dirichlet's Theorem again, we can -after multiplication with a suitable $E_{21}(x)$ from the right- assume for $u=(u',u'')$ that $u'R,u''R$ are both prime ideals in $R$ and coprime to $I.$ Then according to \cite[Lemma~2.1]{gvozdevsky2023width} the quotient $\tilde{R}/u\tilde{R}$ is isomorphic to the product $\mathbb{E}_1\times\mathbb{E}_2$ of the two finite fields $\mathbb{E}_1:=R/u'R$ and $\mathbb{E}_2:=R/u''R$ of even characteristic. But note that the multiplicative groups $\mathbb{E}_i^*$ have odd orders and so the maps
\[
\mathbb{E}_i^*\to\mathbb{E}_i^*,t\mapsto t^2
\]
are isomorphisms. Hence the maps $\mathbb{E}_i\to\mathbb{E}_i,t\mapsto t^2$ are bijections and so all elements of $\tilde{R}/u\tilde{R}$ are squares. So the image $\bar{v}$ in $\tilde{R}/u\tilde{R}$ is a square in $\tilde{R}/u\tilde{R}$. Hence after multiplying with a suitable elementary matrix as in the proof of Claim~\ref{Claim1_bounded_gen_sl}, we have finished the proof of the claim. So using Claim~\ref{square_claim}, one obtains from $A_3$ an $A_2$ of the form 
\[
A_3=
\begin{pmatrix}
1 & 0 & 0 & 0\\
0 & a_3 & 0 & b_3^2\\
0 & 0 & 1 & 0\\
0 & c_3 & 0 & d_3
\end{pmatrix}.
\]
This further enables the use of \cite[Lemma~2]{MR1044049} to obtain an element $A_4$ of 
\[
{\rm Sp}_4^{\alpha}(\tilde{R})=\left\{
\begin{pmatrix}
A & 0\\
0 & A^{-T}
\end{pmatrix}\mid A\in{\rm SL}_2(\tilde{R})
\right\}
\]
after multiplication with $9$ root elements. But the rest of the proof now follows by simply applying \cite[Proposition~2,Lemma~3]{MR1044049} accounting for the remaining $(4*44+4)\Delta=180\Delta$ many root elements. 
\end{proof}

We should note that the above proof can also serve to prove the following absolute bounded generation result instead:

\begin{corollary}\label{symplectic_cor_non_double}
Let $K$ be a global function field, $S$ a non-empty set of valuations of $K$ and $R:=\mathscr{O}_K^S$ the corresponding ring of S-algebraic integers. Further let $n\geq 2$, be given. Then each element of ${\rm Sp}_{2n}(R)$ is a product of $204+4n^2$ many root elements.
\end{corollary}

\begin{remark}
The proof of Lemma~\ref{tilde_bounded_generation_C2} is very similar to the proof of the symplectic case of Kunyavski{\u\i}, Plotkin and Vavilov's \cite[Theorem~A]{kunyavskiui2023bounded}, which I wasn't initially aware of when writing this preprint. The main difference is as follows: The corresponding arguments in \cite{kunyavskiui2023bounded} were done for $\mathbb{F}[T]$ and not the double $\tilde{R}$ of a ring $R$ of all integers in any global function field. This is because, the authors of \cite{kunyavskiui2023bounded} were still lacking some of the arithmetic lemmas which would only be developed later in \cite{https://doi.org/10.1112/blms.12925}, but they did already assemble most of the other ideas needed to prove Corollary~\ref{symplectic_cor_non_double}. In particular, the case-distinction between odd and even characteristic appearing in Claim 3 in the proof of Lemma~\ref{tilde_bounded_generation_C2} is already present in \cite[Lemma~6.4]{kunyavskiui2023bounded} (albeit for $\mathbb{F}[T]$ and not $\tilde{R}$). One could likely also improve many bounds in this paper by more closely adhering to swindling-lemma ideas from \cite{kunyavskiui2023bounded}. Such an approach is followed in Kunyavski{\u\i}, Plotkin and Vavilov's \cite{Kunyavskiı̆_Plotkin_Vavilov_2024}, which appeared after the first version of this preprint.
\end{remark}

\subsubsection{Bounded generation for $G_2$}\label{bounded_gen_g2}

The purpose of this subsubsection is to prove Theorem~\ref{main_thm2_technical} in the $G_2$-case. It is well known that the long roots in the root system $G_2$ form a subroot system isomorphic to $A_2$ and that for any ring $R$ there is consequently a subgroup of $G_2(R)$ isomorphic to $G(A_2,R)={\rm SL}_3(R).$ We note the following:

\begin{proposition}\cite[Theorems~1.3,2.1,4.1]{197877}\label{G_2_stability}
Let $R$ be a ring satisfying the condition ASR$(3)$. Then each element of $G_2(R)$ is a product of at most $21$ many root elements in $G_2(R)$ and an element of the subgroup ${\rm SL}_3(R)$ of $G_2(R).$
\end{proposition}

We will skip the proof as it amounts to simply counting the number of root elements in the various proofs in \cite{197877}. Next, we note the following combination of results by Stein and Estes-Ohm:

\begin{theorem}\cite[Theorem~2.3]{ESTES1967343}\cite[Lemma~2.3]{STEIN1971140}\label{absolute_stable_rank}
Let $R$ be a commutative ring and $I$ an ideal in $R$. Assume further that the space of maximal ideals of $R$ with respect to the Zariski topology is a noetherian space of dimension $d\in\mathbb{N}.$ Then $\tilde{R}$ with respect to $I$ satisfies ASR$(d+2)$.
\end{theorem}

But as noted in \cite[P.~102]{MR0249491}, the Krull-dimension of a ring gives an upper bound on the dimension of the space of maximal ideals. Given that the Krull-dimension of a Dedekind domain is $1,$ we obtain from Proposition~\ref{G_2_stability}, Theorem~\ref{absolute_stable_rank} and Lemma~\ref{bounded_gen_double_implies_conj_gen} that:

\begin{proposition}\label{k1_surjective_stability_g2}
Let $R$ be a Dedekind domain, $I$ a non-zero ideal in $R$. Assume further that $E(A_2,R)=G(\Phi,R)$ and that there is a constant $K\in\mathbb{N}$ with $\|E(A_2,\tilde{R})\|_E\leq K$ for the ring $\tilde{R}$ associated with $R$ and $I.$ Then $\|\bar{E}(G_2,I)\|_{Z(I,G_2)}\leq K+21$ holds. Further, each element of $C(G_2,I)$ is a product of $21$ elements of $Z(I,G_2)$ and an element of $C(A_2,I).$
\end{proposition}

This in combination with the $A_2$-case of Proposition~\ref{bounded_gen_tilde} settles the $G_2$-case of Theorem~\ref{main_thm2_technical}.

\begin{remark}
\hfill
\begin{enumerate}
\item For completeness, we should state that one can show using an idea from \cite{MR1044049} and the main result from \cite{https://doi.org/10.1112/blms.12925} that each element of $G_2(R)$ for $R$ a ring of S-algebraic integers in a global function field is a product of $62+18=80$ many root elements.
\item We would also like to point out that the stability arguments going into Proposition~\ref{k1_surjective_stability_g2} are also similar to arguments for the Dedekind case in \cite[Section~5]{kunyavskiui2023bounded}.
\end{enumerate}
\end{remark}

\subsection{Stability}

The proof of Proposition~\ref{stability} relies mostly on the following stability property that allows to rewrite large products of root elements to a product of just two modulo conjugation in $G(\Phi,R)$:

\begin{lemma}\label{stability_relative_ideal}
Let $R$ be a Dedekind domain and $n\geq 2.$ Further, assume that $v\in R^n$. Then there is an element $A\in{\rm SL}_n(R)$ such that $A\cdot v\in Re_1\oplus Re_2.$
\end{lemma}

\begin{proof}
We will show that Dedekind domains have ideal stable range at most $2$. This suffices to prove the claim of the lemma: For $n=2$ there is nothing to show, so let $n\geq 3$ be given and let $v_1,\dots,v_n$ generate the ideal $I$. By the ideal stable range condition there exist $x_1,\dots,x_n\in R$ such that $v_1':=v_1+x_1\cdot v_n,\dots,v_{n-1}':=v_{n-1}+x_nv_n,$ also generate $I$. Then we can choose (as $v_n$ is an element of $I$) elements $y_1,\dots,y_{n-1}\in R$ such that 
$$
v_n=\sum_{i=1}^{n-1} y_i\cdot v_i'.
$$
But passing from $v:=(v_1,\dots,v_{n-1},v_n)^T$ to $v':=(v_1',\dots,v_{n-1}',0)^T$ is equivalent with multiplying $v$ with 
$$
A:=\left(I_n-\sum_{i=1}^{n-1} y_i\cdot e_{n,i}\right)\cdot\left(I_n+\sum_{i=1}^{n-1}x_i\cdot e_{i,n}\right)\in{\rm SL}_n(R)
$$ 
from the left. Iterating this process we can reduce $v\in R^n$ to an element of $Re_1\oplus Re_2.$

So let $m\geq 3$ be given and consider the ideal $I:=(w_1,\dots,w_m)$ in $R.$ We will assume that $I\neq 0,$ as otherwise there is nothing to show. Consider the prime factorization $\prod_{j=1}^L \C Q_j^{k_j}$ of $I$ in $R.$ Next, pick using the Chinese Remainder Theorem an $x_1\in R$ such that for all prime ideals $\C P$ not among the $\C Q_1,\dots,\C Q_L$ but dividing $(w_2,w_3,\dots,w_{m-1})$, one has
\begin{align*}
&x_1\equiv 1\text{ mod }\C P\text{, if }\C P\text{ divides }w_1\\
&x_1\equiv 0\text{ mod }\C P\text{, if }\C P\text{ does not divide }w_1.
\end{align*}
(If $(w_2,\dots,w_{m-1})$ is $(0)$ then one can set $x_2=1$ and $x_1=x_3=x_4=\cdots=x_{m-1}=0$ to prove the required stable range condition.) Furthermore for $j=1,\dots,L$, the element $x_1$ shall satisfy the following conditions
\begin{align*}
&x_1\equiv 1\text{ mod }\C Q_j^{k_j+1}\text{, if }w_1\in\C Q_j^{k_j+1}\\
&x_1\equiv 0\text{ mod }\C Q_j^{k_j+1}\text{, if }w_1\notin\C Q_j^{k_j+1}.
\end{align*}
Next, consider the ideal $J:=(w_1+x_1w_m,w_2,\dots,w_{m-1}).$ Let $\C P$ be a prime divisor of $J$. So $w_2,\dots,w_{m-1}$ are contained in $\C P.$ Assume for contradiction that $\C P$ is not among the $\C Q_j.$ If $w_1\notin\C P,$ then $x_1\in\C P$ by construction and so $w_1\in\C P,$ a contradiction. Hence $w_1\in\C P.$ But this implies that $x_1\equiv 1\text{ mod }\C P$ and so $w_m\in\C P.$ Thus $\C P$ is a divisor of the ideal $(w_1,\dots,w_m)$, that is one of the $\C Q_j,$ a contradiction. Hence $J=\prod_{j=1}^L\C Q_j^{n_j}$ for some $n_j\in\mathbb{N}_0.$ However, note that $J\subset (w_1,\dots,w_m)$ and so $n_j\geq k_j$ for all $j=1,\dots,l.$ If say $n_1>k_1$ were to hold however, then $w_2,\dots,w_{m-1}\in\C Q_1^{k_1+1}.$ If $w_1\notin\C Q_1^{k_1+1}$, then $x_1\in\C Q_1^{k_1+1}$ and so $w_1\in\C Q_1^{k_1+1},$ a contradiction. Hence $w_1\in\C Q_1^{k_1+1}$ and so $x_1\equiv 1\text{ mod }\C Q_1^{k_1+1}$ and so $w_m\in\C Q_1^{k_1+1}.$ Hence the ideal $(w_1,\dots,w_m)$ is actually contained in $\C Q_1^{k_1+1}$, a contradiction. So $J=(w_1,\dots,w_m)$ holds. Hence the ideal $(w_1,\dots,w_m)$ is already generated by the elements $w_1+x_1w_m,w_2,\dots,w_{m-1}.$ This proves the required stable range condition.
\end{proof}

\begin{remark}
As pointed out to me by the reviewer, using more general methods, it is implied by results in \cite{stepanov1989stable} already that the ideal stable range of a Dedekind domain is at most $3.$ This fact could be used in lieu of Lemma~\ref{stability_relative_ideal} to prove a version of Proposition~\ref{stability} with worse bounds ($12(n-2)$ instead of $8(n-2)$ in the case of $X=A$, say).
\end{remark}

\subsubsection{The special linear group}

For this case we need the following lemma resembling the the proof of \cite[Theorem~4.2]{MR0174604}:

\begin{lemma}\label{stability_sln_prep_step1}
Let $R$ be a Dedekind domain and let $I$ be a non-zero ideal in $R.$ Further, let $n\geq 2$ and $v\in R^n$ be given such that $v\equiv e_1\text{ mod }I$ and such that the entries of $v$ generate $R$ as an ideal. Then there is a product $X$ of six elements of $Z(I,A_n)$ such that $X\cdot v=e_1.$ 
\end{lemma}

\begin{proof}
As demonstrated in the proof of \cite[Theorem~4.2]{MR0174604}, there are elements $Y_1,Y_2,Y_3$ which are all conjugates (or transpose-inverses) of three different elements of the form $X(w):=I_n+\sum_{i=1}^{n-1} w_ie_{i+1,1}$ with all $w_i\in I$ such that $Y_3Y_2Y_1\cdot v=e_1$. Hence we are done if we can show that all such elements $X(w)$ are products of two elements of $Z(I,A_n).$ To this end consider the embedding 
\[
j:{\rm SL}_{n-1}(R)\to{\rm SL}_n(R),A\mapsto
\begin{pmatrix}
1 & \\
 & A
\end{pmatrix}.
\]
Note that one has for $A\in{\rm SL}_{n-1}(R)$ and $w\in R^{n-1}$ that $j(A)\cdot X(w)\cdot j(A)^{-1}=X(A\cdot w).$ So we are done after applying Lemma~\ref{stability_relative_ideal}.  
\end{proof}

This can now be used to derive the $A_n$-version of Proposition~\ref{stability}: Given an element $A\in C(A_n,I),$ one first applies Lemma~\ref{stability_sln_prep_step1} to the first column $v$ of $A.$ Hence after multiplication with six elements of $ Z(I,A_n),$ we can assume that the first column of the new matrix $B\in C(A_n,I)$ is the vector $e_1^T.$ Then multiplying $B$ from the right with $X_2:=I_n+\sum_{i=2}^n b_{1,i}e_{1,i}$, one obtains a matrix $C\in C(A_n,I)$ whose first column and row are $e_1$ and $e_1^T$ respectively. But this implies that $C$ is an element of the subgroup $C(A_{n-1},I)$ of $C(A_n,I)$. Hence we are done by induction, if one can show that $X_2$ is a product of two elements of $Z(I,A_n).$ But this follows again by Lemma~\ref{stability_relative_ideal} as in the proof of Lemma~\ref{stability_sln_prep_step1}.  

\subsubsection{The symplectic group}

This case of Proposition~\ref{stability} requires first the following standard stability property:

\begin{lemma}\label{fixing_first_column_symplectic_step1}
Let $R$ be a Dedekind domain, $n\geq 2$ and let $A\in{\rm Sp}_{2n}(R)$ be given. Then the conjugacy class of $A$ in ${\rm Sp}_{2n}(R)$ contains an element $A'=(a_{ij}')_{1\leq i,j\leq 2n}$ such that the first $n$ entries $a'_{11},\dots,a'_{n1}$ of the first column of $A'$ generate $R$ as an ideal.  
\end{lemma}

\begin{proof}
Note first that the ideal $(a_{11},\dots,a_{2n,1})$ in $R$ is $R.$ But now using the Chinese Remainder Theorem, pick $x_1\in R$ with the following properties: 
\begin{align*}
&x_1\equiv 1\text{ mod }\C P\text{ for all prime divisors }\C P\text{ of }(a_{11},a_{21},\dots,a_{n1},a_{n+2,1},\dots,a_{2n,1}),\\
&x_1\equiv 0\text{ mod }\C P\text{ for all prime divisors }\C P\text{ of }(a_{21},\dots,a_{n1},a_{n+2,1},\dots,a_{2n,1})\text{ not dividing }a_{11}.
\end{align*}
Then consider the ideal $J:=(a_{11}+x_1\cdot a_{n+1,1},a_{21},\dots,a_{n1},a_{n+2,1},\dots,a_{2n,1}).$ If $J$ is not the full ring $R,$ then there is a maximal ideal $\C P$ in $R$ such that $J\subset\C P.$ Then if $a_{11}\in\C P,$ then we obtain that $x_1\equiv 1\text{ mod }\C P$. But $a_{11}+x_1\cdot a_{n+1,1}$ is an element of $\C P$ and so $a_{n+1,1}$ is also an element of $\C P.$ But this is a contradiction as $(a_{11},\dots,a_{n1},a_{n+1,1},a_{n+2,1},\dots,a_{2n,1})=R.$ So $a_{11}\notin\C P.$ But then $x_1\in\C P$ and so $a_{11}\in\C P.$ But this is a contradiction as well. Hence $J=R.$ Next, note that $A_1:=(I_{2n}+x_1\cdot e_{1,n+1})\cdot A\cdot (I_{2n}-x_1\cdot e_{1,n+1})$ is a conjugate of $A$ and all entries of its first column besides the first one $a_{11}^{(1)}=a_{11}+x_1a_{n+1,1}$ agree with the corresponding entries of $A.$ Further, as we have just seen the entries $a_{11}^{(1)},a_{21},\dots,a_{n1},a_{n+2,1},\dots,a_{2n,1}$ of the first column of $A_1$ generate $R.$ Next, one can similarly find an $x_2\in R$ such that the entries $a_{11}^{(1)},a_{21}^{(2)}:=a_{21}+x_2\cdot a_{n+2,1},\dots,a_{n1},\dots,a_{n+3,1},\dots,a_{2n,1}$ of the first column of the conjugate $A_2:=(I_{2n}+x_2\cdot e_{2,n+2})\cdot A_1\cdot (I_{2n}-x_2\cdot e_{2,n+2})$ of $A$ generate $R.$ This process can be iterated until we reach $A_n,$ which we denote by $A'$ and which has the desired property.
\end{proof}

Next, we can derive the following:

\begin{lemma}\label{fixing_first_column_symplectic_step2}
Let $R$ be a Dedekind domain and let $I$ be a non-zero ideal in $R.$ Further, let $n\geq 2$ be given and $A\in C(C_n,I)$ be given. Then up to multiplication with $15$ elements of $Z(I,C_n)$, the first and $n+1.$th columns of $A$ and the first and $n+1$.th row of $A$ can be assumed to be $e_1, e_{n+1}, e_1^T$ and $e_{n+1}^T$ respectively. 
\end{lemma}

\begin{proof}
We first note that by applying Lemma~\ref{fixing_first_column_symplectic_step1}, we can assume that the first $n$ entries $a_{11},\dots,a_{n1}$ of the first column of $A$ already generate the ring $R.$ The vector $v:=(a_{11},\dots,a_{n1})^T$ further satisfies $v\equiv e_1^T\text{ 
 mod } I$. Thus applying Lemma~\ref{stability_sln_prep_step1} to the ${\rm SL}_n(R)$ contained in ${\rm Sp}_{2n}(R)$ via
\[
{\rm SL}_n(R)\to{\rm Sp}_{2n}(R),M\mapsto
\begin{pmatrix}
M & 0_n\\
0_n & M^{-T}.
\end{pmatrix}
\]
and $v$ we can find an element $X$ which is a product of six elements of $Z(I,C_n)$ such that $B:=X\cdot A$ has first column with its first $n$ entries equal to $e_1^T\in R^n.$ Then consider the product
\begin{align*}
C:=&(I_{2n}-b_{2n,1}(e_{n+1,n}+e_{2n,1}))\cdots(I_{2n}-b_{n+3,1}(e_{n+1,3}+e_{n+3,1}))\cdot(I_{2n}-b_{n+2,1}(e_{n+1,2}+e_{n+2,1}))\\
&\cdot(I_{2n}-b_{n+1,1}e_{n+1,1})\cdot B
\end{align*}
and note that its first column is precisely $e_1\in R^{2n}.$
Next, we will show that the product 
\[
(I_{2n}-b_{2n,1}(e_{n+1,n}+e_{2n,1}))\cdots(I_{2n}-b_{n+3,1}(e_{n+1,3}+e_{n+3,1}))\cdot(I_{2n}-b_{n+2,1}(e_{n+1,2}+e_{n+2,1}))
\]
is a product of two elements from $Z(I,C_n)$. To this end, consider the subgroup ${\rm SL}_{n-1}(R)$ of ${\rm Sp}_{2n}(R)$ given by the map
\[
j:{\rm SL}_{n-1}(R)\to{\rm Sp}_{2n}(R),M\mapsto
\begin{pmatrix}
1 & & &\\
& M^{-T} & & \\
& & 1 & \\
& & & M.
\end{pmatrix}.
\]
Note that for an element $v\in R^{n-1}$ and the matrix $Y(v)\in R^{2n\times 2n}$ of the form 
\[
Y(v)=I_{2n}+\sum_{i=1}^{n-1} v_i\cdot (e_{n+i+1,1}+e_{n+1,i+1})
\]
conjugation of $Y(v)$ by $j(M)$ acts on $Y(v)$ by the standard operation of ${\rm SL}_{n-1}(R)$ on $R^{n-1},$ that is $j(M)Y(v)j(M)^{-1}=Y(Mv)$ holds for all $M\in{\rm SL}_{n-1}(R)$ and $v\in R^{n-1}.$ Hence it suffices to show that for any vector $v\in R^{n-1}$ there is some $M\in{\rm SL}_{n-1}(R)$ with $M\cdot v\in Re_1\oplus Re_2.$ But this is precisely the claim of Lemma~\ref{stability_relative_ideal}. So after multiplying $A$ with $9$ elements of $Z(I,C_n),$ we can assume that the first column of $C$ is $e_1.$ 

The $1$ in the first column of $C$ can now be used to clear out the first row of $C$: Picking suitable elements $x_2,\dots,x_n,y_{n+1},\dots,y_{2n}\in I$ one obtains that the product
\begin{align*}
C&\cdot(I_{2n}+x_2(e_{12}-e_{n+2,n+1}))\cdot(I_{2n}+x_3(e_{13}-e_{n+3,n+1}))\cdots(I_{2n}+x_n(e_{1,n}-e_{2n,n+1}))\cdot(I_{2n}+y_{n+1}e_{1,n+1})\\
&\cdot (I_{2n}+y_{n+2}(e_{1,n+2}+e_{2,n+1}))\cdot (I_{2n}+y_{n+3}(e_{1,n+3}+e_{3,n+1}))\cdots(I_{2n}+y_{2n}(e_{1,2n}+e_{n,n+1}))
\end{align*}
called $D$, has the first row equal to $e_1^T.$ However, as before, one obtains using Lemma~\ref{stability_relative_ideal} that the two products 
\begin{align*}
&(I_{2n}+x_2(e_{12}-e_{n+2,n+1}))\cdot(I_{2n}+x_3(e_{13}-e_{n+3,n+1}))\cdots(I_{2n}+x_n(e_{1,n}-e_{2n,n+1}))
\text{ and }\\
&(I_{2n}+y_{n+2}(e_{1,n+2}+e_{2,n+1}))\cdot (I_{2n}+y_{n+3}(e_{1,n+3}+e_{3,n+1}))\cdots(I_{2n}+y_{2n}(e_{1,2n}+e_{n,n+1}))
\end{align*}
are products of two elements of $Z(I,C_n)$ respectively. So in total after multiplying $A$ with $9+4+1=15$ elements of $Z(I,C_n),$ we have reduced $A$ to an element $D$ with first column $e_1$ and first row $e_1^T.$ But this actually puts $D$ into the required form: The columns of $D$ form a symplectic basis wrt the symplectic form given by the matrix $J\in R^{(2n)\times (2n)}.$ Hence note that one has for the k.th column $v_k$ of $D$ that $0=e_1^T\cdot J\cdot v_k=d_{n+1,k}$ if $k\neq n+1.$ Hence the $n+1.$th row of $D$ is $d_{n+1,n+1}e_{n+1}^T=e_{n+1}^T.$ 

However, $D$ is a symplectic matrix and so for the $n+1.$st column
\[
\begin{pmatrix}
w_1\\
w_2
\end{pmatrix}.
\]
of $D$ with $w_1,w_2\in R^n$, the first row of its inverse $E$ is given by $(w_2^T,-w_1^T)$. But using that $D$ is symplectic the $n+1$.th column of $E$ is given by $e_{n+1}.$ Then as before, one can use the fact that the columns of $E$ form a symplectic basis wrt $J$ to deduce that the first row of $E$ is equal to $e_1$. Hence the $n+1.$th column of $D$ is $e_{n+1}$ and we are done. 
\end{proof}

This can now be used to derive the symplectic case of Proposition~\ref{stability}: Applying Lemma~\ref{fixing_first_column_symplectic_step2}, we can assume modulo multiplication with $15$ elements of $Z(I,C_n)$ that the first and $n+1$.th column of $A$ are $e_1$ and $e_{n+1}$ respectively and the first and $n+1.$th row of $A$ are $e_1^T$ and $e_{n+1}^T$. But the set of elements of ${\rm Sp}_{2n}(R)$ with these rows and columns form a subgroup of ${\rm Sp}_{2n}(R)$, namely ${\rm Sp}_{2(n-1)}(R)$. Hence one can by induction reduce any element of $C(C_n,I)$ to an element of $C(C_2,I)$ by multiplying with $15(n-2)$ many elements of $Z(I,C_n).$ This finishes the proof of the symplectic case of Proposition~\ref{stability}.

\section{Strong boundedness}\label{strong_bound_section}

In this section, we will use Theorem~\ref{main_thm2} to prove Theorem~\ref{main_thm1}. As is common in results like Theorem~\ref{main_thm1} there is some differences between $\Phi=C_2$ or $G_2$ and the other cases. The first result we need is the following:

\begin{proposition}\label{plucking_lemma}
Let $R$ a Dedekind domain, $\Phi\in\{A_{n\geq 2},C_{n\geq 2},G_2\}, B\in G(\Phi,R)$, $t$ a level generator of $B$ and $y\in R$ be given. 
Then there is a constant $D:=D(\Phi)\in\mathbb{N}$ independent of $R,B$ and $t$ such that, one has
\begin{equation}\label{root_element_inclusion}
\left\{A\varepsilon_{\phi}(x)A^{-1}\mid A\in G(\Phi,R),x\in I,\phi\in\Phi\right\}\subset B_B(D)
\end{equation}
for $I$ an ideal 
\begin{enumerate}
\item whose radical contains $t$ if $\Phi\in\{A_{n\geq 2},C_{n\geq 3}\}$.
\item containing $(y^2+y)t^2R,$ if $\Phi=C_2.$
\item containing$ (y^2+y)t^3R,$if $\Phi=G_2.$
\end{enumerate}
Further, if $n\geq 3$, then $D(A_n)\leq 32$ and if $n\geq 4$, then $D(C_n)\leq 1920.$
\end{proposition}

The proof can be found in the Appendix and is very similar to arguments from the earlier papers \cite{KLM,explicit_strong_bound_sp_2n,General_strong_bound}. The only real difference is the fact that the calculations can be done avoiding the principal ideal domain assumption by using Lemma~\ref{stability_relative_ideal} instead of Hessenberg matrices. However, the bounds in Proposition~\ref{plucking_lemma} are not optimized and even the presented argument in the appendix could be made to produce slightly better bounds.  We also need the following famous result due to Bass, Milnor and Serre (in the case of ${\rm Sp}_{2n}$ and ${\rm SL}_n$) and Matsumoto (for the other root systems). To state it, we recall the following notation from \cite{MR244257}: For a natural number $n$ and a real number $x$, set $[x]_{[0,n]}:=\inf\{\sup\{0,[x]\},n\}$ for the Gau{\ss}-bracket $[x].$ Then

\begin{theorem}\cite[Corollary~4.3,Theorem~12.4]{MR244257}\cite[Corollaire~4.6]{MR0240214}\label{BMS_thm}
Let $K$ be a global field, $S$ a non-empty set of valuations of $K$ containing all archimedean valuations of $K$ and $\Phi$ an irreducible root system of rank at least $2.$ Further, let $R:=\mathscr{O}_K^S$ be the corresponding ring of S-algebraic integers in $K$ and $I$ a non-zero ideal in $R.$ Then the group $C(\Phi,I)/\bar{E}(\Phi,I)$ is trivial, if $K$ is not a totally imaginary number field or $S$ contains a non-archimedean valuation. Otherwise, $C(\Phi,I)/\bar{E}(I,\Phi)$ is a quotient of the cyclic group $C_{m(I)}$ for $m(I)$ the natural number defined as follows: The exponent $v_p(m(I))$ of the rational prime $p$ in the prime factorization of $m(I)$ is 
\[
\min_{\C P\text{ prime dividing }p}\left[\frac{v_{\C P}(I)}{v_{\C P}(p)}-\frac{1}{p-1}\right]_{[0,v_p(m)]}
\]
for $m:=m(K)$ the number of roots of unity in $K.$ In particular, if $I$ is coprime to the number $m(K),$ then $C(\Phi,I)/\bar{E}(\Phi,I)$ is trivial also for totally imaginary number fields.
\end{theorem}

Third, we need the following elementary lemma:

\begin{lemma}\label{adic_chin_remainder}
Let $K$ be a global field and $S$ a non-empty, finite set of valuations containing all archimedean valuations of $K.$ Further, let $I$ be a proper, non-zero ideal in the ring $R:=\mathscr{O}_K^S$ of S-algebraic integers in $K$ and we define $\|\cdot\|_I:R\to[0,+\infty),x\mapsto 2^{-v_I(x)}$ by setting $v_I(0):=-\infty$ and $v_I(x):=\max\{k\in\mathbb{N}_0\mid x\in I^k\}.$ Then the completion $R_I$ of the ring $R$ wrt $\|\cdot\|_I$ is isomorphic as a topological ring to the direct product ring $\prod_{\C P|I} R_{\C P}$ for $R_{\C P}$ the local ring of integers of the local field associated with the prime $\C P.$ This isomorphism is further induced by the diagonal map $R\to \prod_{\C P|I} R_{\C P}.$
\end{lemma}

We will skip the proof as it is essentially a version of the Chinese Remainder Theorem. Fourth, we need:

\begin{lemma}\label{F_2_congruence_lemma}
Let $K$ be a local field with ring of integers $R$ and maximal ideal $\C Q$ with $R/\C Q=\mathbb{F}_2.$ Further, assume $\Phi=C_2$ or $G_2$ and let $A\in G(\Phi,R)=:G$ be given which is not an element of $C(\Phi,\C Q).$ Then there is a constant $E\in\mathbb{N}$ independent of $K,\Phi$ and $A$, such that there exists a commutator $A':=(A,Y)$ in $G(\Phi,R)$ with the property $C(\Phi,\C Q)\subset B_{A'}(E)$.
\end{lemma}

\begin{proof}
First, we note that there is an element $Y\in G(\Phi,R)$ such that $(A,Y)\notin C(\Phi,\C Q)$ holds. If not, then $A$ would map to a central element in the finite group $G(\Phi,R/\C Q)=G(\Phi,\mathbb{F}_2)$. But this finite group has a trivial center and so $A\in C(\Phi,\C Q),$ a contradiction. Thus in the following, it suffices to show that there is a $D\in\mathbb{N}$ such that for each $B\notin C(\Phi,\C Q)$, one has $C(\Phi,\C Q)\subset B_B(D)$. But now applying a first order compactness argument to the proof of \cite[Theorem~4]{explicit_strong_bound_sp4_pseudo_good}, one can show that there is a constant $D_1\in\mathbb{N}$ independent of $R, K$ and $B$ such that the set $\left\{\varepsilon_{\phi}(x(x-1)y)\mid y\in R,\phi\in\Phi\right\}$ is contained in $B_B(D_1)$ for all $x\in R.$ So picking an element $x\in R^*$ with $v_{\C Q}(x-1)=1$, one obtains that $x(x-1)R=\C Q$. Hence we obtain that $Z(\C Q,\Phi)$ is contained in $B_B(D_1).$ Last, we consider the ring $\tilde{R}:=\tilde{R}_{\C Q}$. Note that $R$ is a local ring and so by \cite[Corollary~2.6]{STEIN1971140} the ring $\tilde{R}$ is semi-local. Hence $\tilde{R}$ has Bass stable range 1. But for such rings, it is known that each element of $G(\Phi,\tilde{R})$ is a product of $4\cdot|\Phi^+|$ many root elements by \cite[Theorem~1]{MR2822515}, that is at most $24$ many root elements. But then by Lemma~\ref{bounded_gen_double_implies_conj_gen}, one obtains that each element of $C(\Phi,\C Q)$ is a product of $24$ elements of $Z(\C Q,\Phi).$ Thus we are done setting $D:=24D_1.$

\end{proof}

Fifth, we require the following lemma:

\begin{lemma}\label{uniform_bound_local_c2}
Let $K$ be a local field with ring of integers $R$ and maximal ideal $\C Q$ with $R/\C Q\neq\mathbb{F}_2.$ Further, let $\Phi=C_2$ or $G_2$ be given. Then there is a constant $C_{\Phi}\in\mathbb{N}$ independent of $K$ such that $\Delta_1(G(\Phi,R))\leq C_{\Phi}.$
\end{lemma}

We will skip the proof as it essentially works the same way as the proof of Lemma~\ref{F_2_congruence_lemma} but uses instead the fact that there is an $x\in R$ such that $x(x-1)R=R.$ Sixth, we need the following technical lemma:

\begin{lemma}\label{sp_4_g2_f2}
Let $H\in\{{\rm Sp}_4(\mathbb{F}_2),G_2(\mathbb{F}_2)\}$ be given and $N\in\mathbb{N}$. Further, let $T$ be a normally generating set of $G:=H^N.$ Then $G=B_T(9|T|)$ holds.
\end{lemma}

We will leave the details of the proof as an exercise, but note that the proof uses the following two facts: First, $H$ has only one non-trivial normal subgroup (namely $[H,H]$), which is either isomorphic to the finite simple groups $A_6$ (for ${\rm Sp}_4(\mathbb{F}_2)$) or $^2A_2(\mathbb{F}_9)$ (for $G_2(\mathbb{F}_2)$) as noted in \cite[Chapter~4]{MR3616493}. Secondly, the table on \cite[P.~61]{Karni_paper} implies that $[H,H]$ has covering number at most $4$ in either case. Seventh, we need the following technical lemma:

\begin{lemma}\label{normally_gen_set_adic}
Let $K$ be a non-archimedean local field with ring of integers $R$, maximal ideal $\C Q$ and field of residue $\kappa.$ Further, let $\Phi$ be an irreducible root system of rank at least $2$ and let $T\subset G(\Phi,R)$ be given. Then $T$ normally generates $G(\Phi,R)$ iff the following two conditions are satisfied:
\begin{enumerate}
\item One has $\Pi(T)=\emptyset$.
\item If $\Phi=C_2$ or $G_2$ and if $\kappa=\mathbb{F}_2$, one has that $T$ maps onto a generating set of $\mathbb{F}_2$ under the abelianization homomorphism $G(\Phi,R)\to G(\Phi,\kappa)=G(\Phi,\mathbb{F}_2)\to\mathbb{F}_2.$
\end{enumerate}
In particular, if $T$ normally generates $G(\Phi,R)$, then there is already an $A\in T$ normally generating $G(\Phi,R).$
\end{lemma}

We will skip the proof as this corollary is essentially a special case of \cite[Corollary~3.12]{General_strong_bound}.

We can finally prove Theorem~\ref{main_thm1} and Theorem~\ref{main_thm0} now:

\begin{proof}
We will first explain how to reduce the proof of Theorem~\ref{main_thm1} to the global function field version of Theorem~\ref{main_thm0}. So let us first assume the setup of Theorem~\ref{main_thm1}; that is $F$ is an algebraic extension of a finite field, $C$ an irreducible, nonsingular, geometrically integral, projective curve defined over $F$, $U$ a non-empty Zariski-open set in $C$ defined over $F$. Further, let $R:=F_C[U]=\mathscr{O}_C(U)$ be the ring of $F$-defined regular functions on $U$ and $\Phi$ an irreducible root system as in Theorem~\ref{main_thm1}. 

Let $T$ be a finite normally generating set of $G(\Phi,R)$ and let $X\in G(\Phi,R)$ be given. For the purposes of this part of the proof we will introduce some additional notation: For a subfield $E$ of $F$ such that $C$ and $U$ are already defined over $E$, note that one can consider the ring of $E$-defined regular functions $R_E:=E[U]$. Assuming that $E$ is a sufficiently large finite subfield of $F$, it also holds that $T\cup\{X\}\subset G(\Phi,R_E).$ Note that $\Pi_R(T)=\emptyset$ and so 
\[
\sum_{A\in T} l_R(A)=R.
\]
But by considering an even larger finite field $\mathbb{E}$ contained in $F$, we also obtain that 
\[
\sum_{A\in T} l_{R_{\mathbb{E}}}(A)=R_{\mathbb{E}}
\]
and so $\Pi_{R_{\mathbb{E}}}(T)=\emptyset.$ Next, if $F\neq\mathbb{F}_2$, then by possibly enlarging $\mathbb{E}$ to also not be $\mathbb{F}_2$, we obtain from \cite[Corollary~3.12]{General_strong_bound} that $T$ normally generates $G(\Phi,R_{\mathbb{E}})$ and $X\in G(\Phi,R_{\mathbb{E}}).$ All of this is to say that, independently of whether $F$ is $\mathbb{F}_2$ or not, we can assume that $F$ is a finite field to begin with and hence that $K$ is a global function field. Hence if the upper bounds on $\Delta_k(G(\Phi,\C O_K^S))$ in Theorem~\ref{main_thm0} can be chosen to not depend on the global function field $K$ (which will be the case), then we have reduced Theorem~\ref{main_thm1} to the global function field case of Theorem~\ref{main_thm0}.

For the proof of Theorem~\ref{main_thm0}, we must distinguish a couple different cases and subcases of global fields $K,$ non-empty set of valuations of $K$ containing all the archimedean valuations of $K$ and root systems $\Phi.$

To this end, we start by dealing with the exceptional root systems $\Phi=E_6,E_7,E_8$ and $F_4$ first. Note that \cite[Proposition~3.5]{General_strong_bound} and the proof of \cite[Theorem~3.1]{General_strong_bound} states that there is a constant $C_1\in\mathbb{N}$ independent of $K,S$ and $\Phi$ (assuming that $\Phi=E_6,E_7,E_8,F_4$) such that each root element of $G(\Phi,R)$ is contained in $B_T(C_1\cdot|T|)$ and for this constant $C_1$, one has
$$
\|G(\Phi,R)\|_T\leq C_1\cdot|T|\cdot\|G(\Phi,R)\|_E.
$$
However, after much research \cite{MR704220,MR2357719,MR3892969,Chevalley_positive_char_tentative,https://doi.org/10.1112/blms.12925,Kunyavskiı̆_Plotkin_Vavilov_2024}, it is now known that independent of $K,S$ and $\Phi$ there is a constant $C_2\in\mathbb{N}$ such that $\|G(\Phi,R)\|_E\leq C_2$ holds (assuming again that $\Phi=E_6,E_7,E_8,F_4$). Setting $C:=C_1\cdot C_2$ and $D(K):=0$ here, this proves Theorem~\ref{main_thm0} for exceptional root systems.

So let us now assume that $\Phi=A_{n\geq 2},C_{n\geq 2}$ or $G_2.$ Our first goal is now to produce an ideal $I:=I(T)$ from the set $T$ such that we can apply Theorem~\ref{main_thm2_technical} (or a version of \cite[Theorem~1.3]{Chevalley_positive_char_tentative} in case of $\Phi=C_2$ or $G_2$) to this ideal. We start with the easiest case of $K$ a global function field or having a real embedding. Here we just apply Proposition~\ref{plucking_lemma} directly to a non-central element $B$ of $T$ to obtain a non-zero ideal $I(T)$ such that 
\[
\left\{A\varepsilon_{\phi}(x)A^{-1}\mid A\in G(\Phi,R),x\in I(T),\phi\in\Phi\right\}\subset B_B(D)
\]
holds for the $D:=D(\Phi)$ of Proposition~\ref{plucking_lemma}. So let us assume that $K$ is a totally imaginary number field and let $m(K)$ be the number of roots of unity in $K$ and $p(K)$ the number of prime divisors of $m(K),$ when $m(K)$ is considered as an element of the ring of S-algebraic integers $R.$ Then note that as $\Pi_R(T)=\emptyset$, there is for each prime ideal $\C P$ dividing $m(K)$ an element $B:=B_{\C P}$ of $T$ and a level generator $t_{\C P}$ of this $B_{\C P}$ such that $t_{\C P}$ is not an element of $\C P.$ Hence applying Proposition~\ref{plucking_lemma} to this $B$ and $t$, one has
\begin{equation}
\left\{A\varepsilon_{\phi}(x)A^{-1}\mid A\in G(\Phi,R),x\in I_{\C P},\phi\in\Phi\right\}\subset B_B(D)
\end{equation}
for the constant $D$ from this proposition and for $I_{\C P}$ an ideal 
\begin{enumerate}
\item whose radical contains $t_{\C P}$ if $\Phi\in\{A_{n\geq 2},C_{n\geq 3}\}$.
\item containing $2t_{\C P}^2R,$ if $\Phi=C_2.$
\item containing $2t_{\C P}^3R,$ if $\Phi=G_2.$
\end{enumerate}
Hence we obtain for the ideal $I:=\sum_{\C P\text{ divides }m(K)} I_{\C P}$ that 
\begin{equation}
\left\{A\varepsilon_{\phi}(x)A^{-1}\mid A\in G(\Phi,R),x\in I,\phi\in\Phi\right\}\subset B_T(D\cdot p(K)).
\end{equation}
But note that this $I$ has to be coprime to $m(K)$ by construction for $\Phi\neq C_2$ or $G_2$. In case that $\Phi=C_2$ or $G_2,$ the ideal $I$ has the form $I=2I'$ for $I'$ coprime to $m(K).$ To unify notation in the following, we will set $p(K)=1$, if $K$ is not totally imaginary. Hence
 \begin{align*}
\bar{E}(\Phi,I)\subset B_T(D\cdot p(K)\cdot L_{\Phi})
 \end{align*}
holds for $L_{\Phi}:=L_{\Phi}(K)$ the upper bound on $\|\bar{E}(\Phi,I)\|_{Z(I,\Phi)}$ as in Theorem~\ref{main_thm2_technical} (if $K$ is not totally imaginary or $\Phi\neq C_2$ or $G_2$). If $\Phi=C_2$ and $K$ is totally imaginary, then one obtains such an upper bound $L_{C_2}$ on $\|\bar{E}(C_2,I)\|_{Z(I,C_2)}$ from our result \cite[Theorem~3.13]{Chevalley_positive_char_tentative}. (Technically, the result was only stated for global function fields in our paper, but the proof clearly also works for number fields.) Note that this bound will depend only on the degree $[K:\mathbb{Q}]$ in this case. If $\Phi=G_2$ and $K$ is totally imaginary, then combining Carter-Keller-Paige's \cite[Corollary~3.13]{MR2357719} with Proposition~\ref{G_2_stability}, one obtains a bound $L_{G_2}$ as required depending on $[K:\mathbb{Q}].$ But then applying Theorem~\ref{BMS_thm} for the ideal $I$, we obtain in all cases that 
\begin{equation}\label{main_thm2_eq}
C(\Phi,I)=\bar{E}(\Phi,I)\subset B_T^{(R)}(D\cdot p(K) L_{\Phi}).
\end{equation} 

Next, consider the prime factorization $I=\prod_{i=1}^N\C P_i^{m_i}$ of $I$ in the Dedekind domain $R.$ Note that according to Lemma~\ref{adic_chin_remainder}, we obtain that 
\begin{equation}\label{p_adic_chineseII}
G(\Phi,R_I)=\prod_{i=1}^N G(\Phi,R_{\C P_i}).
\end{equation}
For each $A\in T$, we denote the subset of $\C Q\in\{\C P_1,\dots,\C P_N\}=:M$ such that $A$ normally generates $G(\Phi,R_{\C Q})$ by $\Lambda(A).$ Next, we need the following claim: 
\begin{claim}\label{densely_generated_claim}
The subset $T$ of $G(\Phi,R_I)$ normally generates $G(\Phi,R_I)$.
\end{claim}
The basic idea is to leverage the fact that $G(\Phi,R)$ is dense in $G(\Phi,R_I):$ We first note that as $G(\Phi,R)$ is bounded (cf. \cite{Gal-Kedra-Trost}) there is an $L:=L(\Phi,R,T)\in\mathbb{N}$ and elements $t_1,\dots,t_L\in T\cup T^{-1}$ such that 
\[
G(\Phi,R)=\left\{\prod_{i=1}^L A_i\cdot t_i\cdot A_i^{-1}\mid A_i\in G(\Phi,R)\right\}.
\]
Next, pick an element $X\in G(\Phi,R_I)$ and a sequence of elements $(X^{(n)})_n$ in $G(\Phi,R)$ converging to $X$ in $G(\Phi,R_I).$ But then we can find sequences of elements $(A_i^{(n)})_n$ for $i=1,\dots,L$ with 
\[
X^{(n)}=\prod_{i=1}^L A_i^{(n)}\cdot t_i\cdot (A_i^{(n)})^{-1}.
\]
But note that the group $G(\Phi,R_I)$ is a compact metric space and so passing to suitable subsequences of the $(A_i^{(n)})_n$, we can assume that all of these sequences $(A_i^{(n)})_n$ converge to some $A_i\in G(\Phi,R_I)$ for all $i=1,\dots,L.$ But now the continuity of multiplication in $G(\Phi,R_I)$ clearly implies 
\[
X=\prod_{i=1}^L A_i\cdot t_i\cdot A_i^{-1}
\]
and so $X$ is contained in the subgroup of $G(\Phi,R_I)$ normally generated by $T$. But $X$ was arbitrary and so $T$ normally generates $G(\Phi,R_I)$ as in the claim.

But due Claim~\ref{densely_generated_claim}, we now have 
\[
\bigcup_{A\in T}\Lambda(A)=\{\C P_1,\dots,\C P_N\}=M.
\]
Next, if $\Phi=C_2$ or $G_2$ and $A\in T$, we decompose $\Lambda(A)$ as follows
\[
\Lambda(A)=\{\C Q\in\Lambda(A)\mid R/\C Q\neq\mathbb{F}_2\}\cup\{\C Q\in\Lambda(A)\mid R/\C Q=\mathbb{F}_2\}=:\Lambda_1(A)\cup\Lambda_2(A)
\]
We decompose the set $M$ into $M_1$ and $M_2$ in the same way according to whether $R/\C Q=\mathbb{F}_2$ or not. If $\Phi\neq C_2$ or $G_2$, we set $\Lambda_1(A):=\Lambda(A),M_1:=M$ and $\Lambda_2(A):=\emptyset=:M_2$ for the sake of not repeating arguments. Next, for $A\in T$ and $\C Q\in\Lambda_1(A)$, pick an element $Y_{\C Q}(A)\in G(\Phi,R_{\C Q})$ such that the commutator $Z_{\C Q}(A):=(A,Y_{\C Q}(A))$ is a normal generator of $G(\Phi,R_{\C Q})$. Such a $Y_{\C Q}(A)$ must exist:

If none of the commutators $(A,Y)$ for $Y\in G(\Phi,R_{\C Q})$ were to normally generate $G(\Phi,R_{\C Q})$, then all the commutators $(A,Y)$ map to the center of the finite group $G(\Phi,\kappa)=:\bar{G}$ for $\kappa:=R_{\C Q}/\C Q$ after reducing modulo $\C Q$ by Lemma~\ref{normally_gen_set_adic}. (This is the place where we use $R/\C Q\neq\mathbb{F}_2$.) Phrased differently, for the image $\bar{A}$ of $A$ in $\bar{G}$ and for $\bar{Y}\in\bar{G},$ there is a $Z(\bar{Y})\in Z(\bar{G})$ such that one has $\bar{Y}\bar{A}\bar{Y}^{-1}=\bar{A}\cdot Z(\bar{Y}).$ But $A$ normally generates $G(\Phi,R_{\C Q})$ and hence $\bar{A}$ normally generates $\bar{G}$. So for each $\bar{X}\in\bar{G}$, there is an $L\in\mathbb{N},\epsilon_1,\dots,\epsilon_L\in\{1,-1\}$ and $\bar{Y}_1,\dots,\bar{Y}_L\in\bar{G}$ with 
\[
\bar{X}=\prod_{i=1}^L\bar{Y}_i\bar{A}^{\epsilon_i}\bar{Y}_i^{-1}=\prod_{i=1}^L\bar{A}^{\epsilon_i}Z(\bar{Y}_i)^{\epsilon_i}=\left(\bar{A}^{\sum_{i=1}^L\epsilon_i}\right)\cdot \left(\prod_{i=1}^L Z(\bar{Y}_i)^{\epsilon_i}\right).
\]
But the right hand side is an element of the abelian subgroup $\langle \bar{A}\rangle\cdot Z(\bar{G})$ of $\bar{G}$ and as $\bar{X}$ was arbitrary, this implies that $\bar{G}=G(\Phi,\kappa)$ is an abelian group, which it is clearly not.

The ring $R_{\C Q}$ for $\C Q\in\Lambda_1(A)$ is local and a principal ideal domain. Hence there is a $C_{\Phi}\in\mathbb{N}$ such that $\Delta_1\left(G(\Phi,R_{\C Q})\right)\leq C_{\Phi}$ for all $\C Q\in\Lambda_1(A):$ This $C_{A_n}\leq 12n$ is obtained for $\Phi=A_{n\geq 2}$ from Kedra-Libman-Martin's \cite[Theorem~6.3]{KLM} and $C_{C_n}\leq 576(3n-2)$ for $\Phi=C_{n\geq 3}$ from our result \cite[Theorem~2]{explicit_strong_bound_sp_2n}. For $\Phi=C_2,G_2$ we use Lemma~\ref{uniform_bound_local_c2} instead.

Given that $Z_{\C Q}(A)$ normally generates $G(\Phi,R_{\C Q})$, we hence obtain that 
\[
G(\Phi,R_{\C Q})=B_{Z_{\C Q}(A)}^{(R_{\C Q})}(C_{\Phi}).
\]
Next, for $S_A:=\prod_{\C Q\in\Lambda_1(A)}R_{\C Q}$ pick $U\in G(\Phi,S_A)$ arbitrary. Then using an isomorphism as in (\ref{p_adic_chineseII}), we can find find for all $\C Q\in\Lambda_1(A)$ elements $S^{(\C Q)}_1,\dots,S^{(\C Q)}_{C_{\Phi}}\in G(\Phi,R_{\C Q})$ as well as $\epsilon_1^{(\C Q)},\dots,\epsilon_{C_{\Phi}}^{(\C Q)}\in\{-1,0,1\}$ such that 
\[
U=\left(\prod_{i=1}^{C_{\Phi}} S_i^{(\C Q)} (Z_{\C Q}(A))^{\epsilon_i^{(\C Q)}} (S_i^{(\C Q)})^{-1}\right)_{\C Q\in\Lambda_1(A)}.
\]
But now for $i=1,\dots,C_{\Phi}$ and $\C Q\in\Lambda_1(A),$ set $M_i^{(\C Q)}\in G(\Phi,R_{\C Q})$ as $Y_{\C Q}(A)$ if $\epsilon_i^{(\C Q)}=1$ and as the trivial element $1\in G(\Phi,R_{\C Q})$ if not. Further, set $N_i^{(\C Q)}\in G(\Phi,R_{\C Q})$ as $Y_{\C Q}(A)$ if $\epsilon_i^{(\C Q)}=-1$ and as the trivial element $1$ if not. Then note that 
\[
\left(A,M_i^{(\C Q)}\right)\cdot\left(N_i^{(\C Q)},A\right)=(Z_{\C Q}(A))^{\epsilon_i^{(\C Q)}}
\]
holds. Setting $M_i:=\left(M_i^{(\C Q)}\right)_{\C Q}, N_i:=\left(N_i^{(\C Q)}\right)_{\C Q},S_i:=\left(S_i^{(\C Q)}\right)_{\C Q}\in G(\Phi,S_A),$ we thus obtain
\[
U=\prod_{i=1}^{C_{\Phi}} S_i \left(A,M_i\right)\cdot\left(N_i,A\right) S_i^{-1}.
\]
Hence, we obtain that 
\begin{equation}\label{p-adic-conjugacy-good-primes_prelim}
G(\Phi,S_A)=B_A^{(S_A)}(4C_{\Phi}).
\end{equation}
But combining these results for all $A\in T$, we obtain 
\begin{equation}\label{p-adic-conjugacy-good-primes}
\prod_{\C Q\in M_1}G(\Phi,R_{\C Q})\subset B_T^{(R_I)}(4C_{\Phi}|T|).
\end{equation}
Then similar to the argument for (\ref{p-adic-conjugacy-good-primes_prelim}) and slightly abusing notation, one can using Lemma~\ref{F_2_congruence_lemma} prove for $A\in T$ that 
\begin{equation}\label{p-adic-conjugacy-bad-primes-congruence_prelim}
C\left(\Phi,\prod_{\C Q\in\Lambda_2(A)}\C Q\right)=B_A^{(\prod_{\C Q\in\Lambda_2(A)}R_{\C Q})}(2E).
\end{equation}
But combining these results for all $A\in T$, we obtain 
\begin{equation}\label{p-adic-conjugacy-bad-primes-congruence}
C\left(\Phi,\prod_{\C Q\in M_2} \C Q\right)\subset B_T^{(R_I)}(2E\cdot|T|).
\end{equation}
Consider next the image of an arbitrary $X\in G(\Phi,R)$ in the finite group 
\[
G:=\prod_{\C Q\in M_2}G(\Phi,R/\C Q)=\prod_{\C Q\in M_2}G(\Phi,\mathbb{F}_2).
\]

But according to Lemma~\ref{sp_4_g2_f2} and Claim~\ref{densely_generated_claim}, there is an element $X'\in B_T^{(R_I)}(9|T|)$ such that $X'X^{-1}$ is an element of 
\[
\prod_{\C Q\in M_1}G(\Phi,R_{\C Q})\times C\left(\Phi,\prod_{\C Q\in M_2} \C Q\right).
\]
But now we can combine this with (\ref{p-adic-conjugacy-bad-primes-congruence}) and (\ref{p-adic-conjugacy-good-primes}) to conclude that 
\[
X\in 
\begin{cases}
B_T^{(R_I)}\left((4C_{\Phi}+2E+9)\cdot|T|\right)&\text{ if }\Phi=C_2,G_2\\
B_T^{(R_I)}\left(4C_{\Phi}\cdot|T|\right)&\text{ if }\Phi\neq C_2,G_2
\end{cases}
\subset G(\Phi,R_I).
\]
Denoting the bounds $(4C_{\Phi}+2E+9)$ and $4C_{\Phi}$ both by $V_{\Phi}$ for the remainder of the proof, we see that $X$ is a product of at most $V_{\Phi}\cdot|T|$ many $G(\Phi,R_I)$-conjugates of elements of $T\cup T^{-1}.$ But we can approximate (with respect to the norm on $G(\Phi,R_I)$ given by $I$) each conjugating element from $G(\Phi,R_I)$ in this product by elements from $G(\Phi,R)$ such that replacing the conjugating factors by elements from $G(\Phi,R)$, we obtain an element 
\[
X''\in B_T^{(R)}\left(V_{\Phi}\cdot|T|\right)
\]
such that $X$ and $X''$ are very close with respect to the norm given by $I.$ Hence we obtain $(X'')^{-1}\cdot X\in C(\Phi,I)\subset G(\Phi,R).$ However, by (\ref{main_thm2_eq}), we thus get $(X'')^{-1}X\in B_T(D\cdot p(K) L_{\Phi})$ and so $G(\Phi,R)\subset B_T^{(R)}\left(D\cdot p(K) L_{\Phi}+V_{\Phi}\cdot|T|\right).$ This finishes the proof.
\end{proof}  

\begin{remark}\label{post_thm0_remark}
The proof strategy doesn't actually require local fields and their rings of integers. The main reason to use the completed ring $R_I$ (rather than the perhaps more natural quotient ring $R/I$) is twofold: First, it allows one to avoid considering the multiplicities of the prime divisors of $I.$ More importantly though, the explicit bounds from \cite[Theorem~6.3]{KLM} and \cite[Theorem~2]{explicit_strong_bound_sp_2n} used rely on the underlying ring being an integral domain, which wouldn't necessarily be the case for the direct factors $R/\C Q^{k_{\C Q}}$ of $R/I$.
\end{remark}

\label{bound_discussion}
Going through the proof of Theorem~\ref{main_thm0} and inserting the explicit values for the various constants as available, one can observe the following inequalities (ignoring for the moment the cases of $\Phi=C_2,G_2,E_6,E_7,E_8$ and $F_4$):
\begin{align*}
\Delta_k(G(\Phi,R))&\leq D\cdot p(K) L_{\Phi}+V_{\Phi}\cdot k\\
&\leq
\begin{cases}
32 p(K)\cdot(L_A(K)+8(n-2))+48nk&\text{ ,if }\Phi=A_{n\geq 3}\\
1920 p(K)\cdot(L_C(K)+15(n-2))+2304(3n-2)k&\text{ ,if }\Phi=C_{n\geq 4}\\
D(A_2) p(K)\cdot L_A(K)+96k&\text{ ,if }\Phi=A_2\\
D(C_3) p(K)\cdot (L_C(K)+15)+16128k&\text{ ,if }\Phi=C_3
\end{cases}\\
&\leq
\begin{cases}
32 p(K)\cdot(L_A(K)-16)+256p(K)n+48nk&\text{ ,if }\Phi=A_{n\geq 3}\\
1920 p(K)\cdot(L_C(K)-30)+28800p(K)n+6912nk&\text{ ,if }\Phi=C_{n\geq 4}\\
D(A_2) p(K)\cdot L_A(K)+96k&\text{ ,if }\Phi=A_2\\
D(C_3) p(K)\cdot (L_C(K)+15)+16128k&\text{ ,if }\Phi=C_3
\end{cases}
\end{align*}
for all $k\in\mathbb{N}.$ Next,for $\Phi=A_{n\geq 3},$ one has 
\[
\Delta_k(G(A_n,R))\leq 32 p(K)\cdot(L_A(K)-16)+256p(K)n+48nk=32 p(K)\cdot(L_A(K)-16)+ n(48k+256p(K)).
\]
Obviously, for specific examples of global fields $K$, one can now find explicit bounds on $\Delta_k$, but we will restrict ourselves to the following four:

\begin{corollary}\label{example}
Let $K$ be a global field, $S$ a non-empty set of valuations in $K$ containing all archimedean valuations of $K$ and $R$ the corresponding ring of S-algebraic integers. Further, assume that $K$ is a global function field or that $K$ has a real embedding and $R$ is a principal ideal domain. Then one has for $n\geq 4$ and $k\in\mathbb{N}$ that
\begin{align*}
&\Delta_k({\rm SL}_n(R))\leq 5712+48n(k+6),\Delta_k({\rm SL}_n(\mathbb{Z}))\leq 5712+48n(k+6),\\
&\Delta_k({\rm Sp}_{2n}(R))\leq 1028350+28800n +6912nk\text{ and }\Delta_k({\rm Sp}_{2n}(\mathbb{Z}))\leq 1028350+28800n+6912nk.
\end{align*}
\end{corollary}

For the ring $\mathbb{Z}$ and small values of $k$ and $n$ the bounds in \cite{KLM} on $\Delta_k({\rm SL}_n(\mathbb{Z}))$ (that is $(4n+4)(4n+51)k$) are better than our bounds, but already for $n\geq 23$, we obtain better bounds and obviously our asymptotic in $n$ is better. For $A_2$ and $C_3$ the constants $D(A_2)$ and $D(C_3)$ are not explicitly known from our argument; this however is a technical issue and using slightly different calculations not relying on Lemma~\ref{stability_relative_ideal}, one could make these constants $D(A_2),D(C_3)$ explicit. Now in order to consider $\Phi=C_2$ or $G_2$, define for $\Phi\in\{C_2, G_2\}$ the constant
\begin{align*}
M_{\Phi}:=\min\{i\in\mathbb{N}\mid &\forall\text{ comm. rings }R:\forall A\in G(\Phi,R):\\
&\Pi(A)=\emptyset\Rightarrow (\forall x,y\in R:\forall\phi\in\Phi:\varepsilon_{\phi}(x(x-1)y)\subset B_A(i)) \}
\end{align*}
Note that as seen in the proof of Lemma~\ref{F_2_congruence_lemma} and the omitted one from Lemma~\ref{uniform_bound_local_c2}, one has for the respective values of $E$ and $C_{\Phi}$ that $E\leq 4\cdot M_{\Phi}|\Phi^+|$ and $C_{\Phi}\leq 4\cdot M_{\Phi}|\Phi^+|.$ Hence one obtains for $\Phi=C_2,G_2$ that
\begin{align*}
\Delta_k(G(\Phi,R))\leq D(\Phi)\cdot p(K)\cdot L_{\Phi}(K)+V_{\Phi}\cdot k
\leq D(\Phi)\cdot p(K)\cdot L_{\Phi}(K)+(24M_{\Phi}\cdot|\Phi^+|+9)\cdot k.
\end{align*}
The constants $D(\Phi)$ and $M_{\Phi}$ are non-explicit in our arguments but could be explicitly given by going through the respective proofs and counting (instead of appealing to compactness). The constant $L_{\Phi}(K)$ is explicit (and given by Theorem~\ref{main_thm2_technical}) if $K$ is not totally imaginary, but its value is unclear if $K$ is totally imaginary (but more on that in the next section).

For the exceptional root systems $\Phi=E_6,E_7,E_8$ and $F_4$, the proof of Theorem~\ref{main_thm0} establishes that there is a constant $C\in\mathbb{N}$ independent of the global field $K$ and set of valuations $S$ such that $\Delta_k(G(\Phi,\C O_K^S))\leq C\cdot k$ holds for all $k\in\mathbb{N}.$ The constant $C$ doesn't depend on $\Phi$, because here we are only considering finitely many $\Phi;$ trying to generalize this idea to all root systems though only produces upper bounds on $\Delta_k$ proportional to ${\rm rank}(\Phi)^2\cdot k.$ (This is essentially the "old" approach from \cite{KLM,Chevalley_positive_char_tentative,explicit_strong_bound_sp_2n,General_strong_bound} mentioned in the introduction.)

\section{Discussion of $D(K)$, closing Remarks and open problems}

Besides the various "formal" constants $V_{\Phi}, D(\Phi)$ (and to an extent $L_{\Phi}(K)$) making up $C$ from Theorem~\ref{main_thm0}, the most interesting question remaining after the proof of Theorem~\ref{main_thm0} is the precise value of $D(K).$ From the proof, one sees first that $D(K)$ has an upper bound proportional to $p(K)\cdot\max\{L_A(K),L_C(K),L_{C_2}(K),L_{G_2}(K)\}$ for $L_A(K)$ and $L_C(K)$ as in Theorem~\ref{main_thm2_technical} and with $L_{C_2}(K),L_{G_2}(K)$ constructed only for totally imaginary number fields in the proof of Theorem~\ref{main_thm0} with the proportionality factor independent of $K,S$ and $\Phi.$ 

Ideally, one would hope that $D(K)$ in Theorem~\ref{main_thm0} (and so the upper bound on $\Delta_k(G(\Phi,\C O_K^S))$) can be chosen independent of the global field $K$ entirely, but the extent to which the current proof strategy allows one to do so depends on the global field. First, note that if $K$ is any global field other than a totally imaginary number field, one has $p(K)=1$. Further, for global function fields $L_A(K)=215,L_C(K)=566$ -and hence $D(K)$- are independent of $K$.
\begin{enumerate}
\item For real number fields, one has 
\begin{align*}
\max\{L_A(K),L_C(K)\}=\max\{68\cdot\Delta(K)+11,180\cdot\Delta(K)+26\}=180\cdot\Delta(K)+26
\end{align*}
for $\Delta(K)$ the number of ramified primes of $K|\mathbb{Q}.$ 
\item If $K$ is a totally imaginary number field than $D(K)$ depends strongly on the number field: On the one hand, in this case $p(K)$ and $\Delta(K)$ can become arbitrary large (say by considering cyclotomic fields). Further, 
\begin{align*}
\max\{L_A(K),L_C(K),L_{C_2}(K),L_{G_2}(K)\}&=\max\{68\cdot\Delta(K)+11,180\cdot\Delta(K)+26,L_{C_2}(K),L_{G_2}(K)\}\\
&=\max\{180\cdot\Delta(K)+26,L_{C_2}(K),L_{G_2}(K)\}
\end{align*}
depends on $K$ beyond that: The constants $L_{C_2}(K)$ and $L_{G_2}(K)$ are -as stated in the proof of Theorem~\ref{main_thm0}- both bounded above by $g([K:\mathbb{Q}])$ for $g:\mathbb{N}\to\mathbb{N}$ \textit{some unknown} function.
\end{enumerate}

However, these dependencies on $\Delta(K)$ and $[K:\mathbb{Q}]$ in Theorem~\ref{main_thm0} for number fields are likely an artifact of the proof strategy for Theorem~\ref{main_thm2_technical}: It seems likely that for a non-zero ideal $I$, one can -by combining the idea of considering the double from \cite{gvozdevsky2023width} with the arguments from \cite{MR3892969} prove the following conjecture:

\begin{conjecture}\label{foundational_conjecture}
Let $K$ be a global field, $S$ a finite set of valuations of $K$ containing all archimedean valuations of $K$ and $I$ a non-zero ideal in $R:=\C O_K^S$ and $\Phi$ an irreducible root system of rank at least $2.$ Then $\|\bar{E}(\Phi,I)\|_{Z(I,\Phi)}$ has an upper bound only depending on $|C(\Phi,I)/\bar{E}(\Phi,I)|$ but not on $K,I,S.$
\end{conjecture}

If $K$ is a real number field or $K$ is totally imaginary, but the set $S$ contains a finite prime than by Theorem~\ref{BMS_thm}, one has $|C(\Phi,I)/\bar{E}(\Phi,I)|=1$ and using Conjecture~\ref{foundational_conjecture} instead of Theorem~\ref{main_thm2_technical}, \cite[Theorem~3.13]{Chevalley_positive_char_tentative} and \cite[Corollary~3.13]{MR2357719} in the proof of Theorem~\ref{main_thm0}, one can achieve that $D(K)$ only depends on $p(K)$ in those two cases. However, using Conjecture~\ref{foundational_conjecture} instead of Theorem~\ref{main_thm2_technical}, $p(K)$ only has to be introduced in the proof of Theorem~\ref{main_thm0} to ensure that the ideal $I$ satisfies $|C(\Phi,I)/\bar{E}(\Phi,I)|=1.$ But according to Theorem~\ref{BMS_thm} this holds for these two classes of global fields $K$ and sets $S$ anyway. So assuming Conjecture~\ref{foundational_conjecture}, one can produce a constant $D(K)$ independent of $K$ in these two cases.

Unfortunately, in the totally imaginary number field case, when $S$ only contains archimedean valuations, while Conjecture~\ref{foundational_conjecture} can still be used to ensure that $D(K)$ only depends on $p(K)$, the current strategy for Theorem~\ref{main_thm0} doesn't enable one to eliminate this dependency in general. That being said, if one were to only consider finitely many root systems $\Phi$ (rather than entire infinite families), Conjecture~\ref{foundational_conjecture} is useful for totally imaginary number fields still: For example for $\Phi=C_2,G_2$, one has that $Z(2\C O_K^S,\Phi)$ is contained in $B_T(Z\cdot|T|)$ for some integer $Z.$ But by Theorem~\ref{BMS_thm}, one has $|C(\Phi,2\C O_K^S)/\bar{E}(\Phi,2\C O_K^S)|=1$ and so using Conjecture~\ref{foundational_conjecture} in the proof of Theorem~\ref{main_thm0}, one can also achieve that $D(K)$ is independent of $K$ for $\Phi=C_2,G_2$ in the totally imaginary number field case.

We want to point out that for $\Phi=A_{n\geq 2},C_{n\geq 4}$ and all global fields $K$ and rings of S-algebraic integers $R$, one can (combining the arguments from this paper and \cite{KLM,explicit_strong_bound_sp_2n}) produce for all $k\in\mathbb{N}$ upper bounds on $\Delta_k(G(\Phi,R))$ of the form $Z'\cdot{\rm rank}(\Phi)^2k$ for an explicit number $Z'$. As we saw in Corollary~\ref{example} for global function fields there is linear bounds (in $n$) on $\Delta_k$ for all $k$ and so we ask instead for number fields: 

\begin{conjecture}
For $n\in\mathbb{N}$ define the number
\[
F(n):=\sup\left\{\frac{\Delta_k({\rm SL}_{n+2}(\mathscr{O}_K))}{k}\mid K\text{ a number field },k\in\mathbb{N}\right\}.
\]
Does $F:\mathbb{N}\to[0,+\infty),n\mapsto F(n)$ grow faster than linear?
\end{conjecture}
Clearly $n+2\leq F(n)\leq Z'(n+1)^2$ holds, but essentially the question is whether the number field $K$ is relevant at all for the generalized conjugacy diameters. There remains also the case of ${\rm SL}_2(\C O_K^{S})$ for $(\C O_K^S$ a ring of S-integers with infinitely many units. The main problem here is to generalize the bounded generation results from \cite{gvozdevsky2023width} similar to what is done for $G_2$ and ${\rm Sp}_4$ in the current paper. Disregarding these arithmetic difficulties, it would further be interesting whether one can determine better bounds on, say, $\Delta_k({\rm SL}_n(\mathscr{O}_K^S))$. A first straightforward way of improving even our strategy it is to produce better upper bounds on $\Delta_1({\rm SL}_n(R))$ for a discrete valuation domain $R$ than the bound $12(n-1)$ from \cite{KLM}. A second way is finding a better method to reduce the integer case to the local case, besides the two-commutator trick we use in the proof of Theorem~\ref{main_thm0}.

Last, we should mention that while working on this preprint Chen Meiri and Nir Avni published a preprint which provides bounds on the conjugacy width of anisotropic orthogonal groups \cite{avni2023conjugacy} and it would be interesting to see whether the methods they use can be used to produce bounds on $\Delta_k$ for isotropic arithmetic groups in general. On this note, there is also the more practical question of bounds on $\Delta_k(\Gamma)$ for groups $\Gamma$ commensurable with $G(\Phi,\mathscr{O}_K^S).$ 

\section*{Appendix}\label{appendix}

We write down the proof of Proposition~\ref{plucking_lemma}.

\begin{proof}
For $\Phi=A_2,C_2,G_2,C_3$, we simply refer to our result \cite[Theorem~3.2]{General_strong_bound} and its proof. 

We next consider $\Phi=A_n$ with $n\geq 3.$ Then after conjugating by suitable Weyl-group elements in ${\rm SL}_{n+1}(R)$, we may assume that $t=b_{21}$ or $t=b_{11}-b_{22}.$ In the first case, note that we may using the proof of Lemma~\ref{stability_relative_ideal} assume further that $b_{4,1}=b_{5,1}=\cdots=b_{n+1,1}=0.$ Then using the double commutator formula \cite[Lemma~6.7]{KLM} for some $j\neq 1,4$, we obtain for $x\in R$ and setting $C:=B^{-1}$ that
\begin{align*}
&D:=\left((B,E_{1j}(1)),E_{j4}(x)\right)=I_{n+1}-xe_{14}+x(c_{jj}-c_{j1})Be_{14}\\
&=I_{n+1}+
\begin{pmatrix}
0 & 0 & 0 & x((c_{jj}-c_{j1})b_{11}-1) & 0 & \cdots & 0\\
0 & 0 & 0 & x(c_{jj}-c_{j1})b_{21} & 0 & \cdots & 0\\
0 & 0 & 0 & x(c_{jj}-c_{j1})b_{31} & 0 & \cdots & 0\\
0 & 0 & 0 & 0 & 0 & \cdots & 0\\
\cdot & \cdot & \cdot & \cdot & \cdot & \cdots & \cdot\\
0 & 0 & 0 & 0 & 0 & \cdots & 0\\
\end{pmatrix}
\end{align*}
So considering the two commutators $(E_{21}(1),D)=E_{24}(x((c_{jj}-c_{j1})b_{11}-1))$ and $(E_{32}(1),D)=E_{34}(x(c_{jj}-c_{j1})b_{21}),$ we obtain for $I:=((c_{jj}-c_{j1})b_{11}-1,(c_{jj}-c_{j1})b_{21})$ that 
\[
\left\{A\varepsilon_{\phi}(x)A^{-1}\mid A\in G(A_n,R),x\in I,\phi\in A_n\right\}\subset B_B(8).
\]
However, the radical of $I$ has the desired property: To show this it suffices that each maximal ideal $\C P$ of $R$ containing $I$, contains $t=b_{12}.$ So let $\C P$ be such a maximal ideal. If $c_{jj}-c_{j1}\in\C P,$ then $1$ is an element of $I$ and so of $\C P$, a contradiction. Hence $c_{jj}-c_{j1}$ is not in $\C P.$ But given that $(c_{jj}-c_{j1})b_{21}\in I\subset\C P,$ we obtain $t=b_{21}\in\C P$ and so we are done with the first case of $t=b_{21}.$ But note further that in the second case, we obtain that 
\begin{align*}
&B':=\left(B,E_{12}(1)\right)=(I_{n+1}+Be_{12}B^{-1})\cdot (I_{n+1}-e_{12})\\
&=I_{n+1}+Be_{12}C-e_{12}-Be_{12}Ce_{12}=I_{n+1}+Be_{12}C-e_{12}-c_{21}Be_{12}\\
&=I_{n+1}+
\begin{pmatrix}
b_{11}c_{21} & b_{11}(c_{22}-c_{21})-1 & b_{11}c_{23} & \cdots & b_{11}c_{2,n+1}\\
b_{21}c_{21} & b_{21}(c_{22}-c_{21}) & b_{21}c_{23} & \cdots & b_{21}c_{2,n+1}\\
\cdot & \cdot & \cdot & \cdot & \cdot\\
b_{n+1,1}c_{21} & b_{n+1,1}(c_{22}-c_{21}) & b_{n+1,1}c_{23} & \cdots & b_{n+1,1}c_{2,n+1}
\end{pmatrix}
\end{align*}
But now running through the argument of the first case for $B$, a conjugate of $B$ and a suitable conjugate of $B'$, we find an ideal $I$ in $R$ such that
\[
\left\{A\varepsilon_{\phi}(x)A^{-1}\mid A\in G(A_n,R),x\in I,\phi\in A_n\right\}\subset B_B(2*8+16)=B_B(32)\text{ and }
\]
\begin{equation}\label{eq1}
B\equiv 
\begin{pmatrix}
b_{11} & 0 & * & \cdots & *\\
0 & b_{22} & * &\cdots & *\\
0 & 0 & * & \cdots & *\\
\cdot & \cdot & \cdot & \cdot & \cdot\\
0 & 0 & * & \cdots & *
\end{pmatrix}\text{ mod }\sqrt{I}
\end{equation}
holds and $b_{11}(c_{22}-c_{21})-1=b'_{21}\in \sqrt{I}$. Further, we obtain $c_{21}\in \sqrt{I}$ from $C=B^{-1}$ and (\ref{eq1}). Hence $b_{11}c_{22}-1\in \sqrt{I}$ and so using $C=B^{-1}$, we obtain that $b_{11}\equiv c_{22}^{-1}\equiv b_{22}\text{ mod }\sqrt{I},$ that is $t=b_{11}-b_{22}\in \sqrt{I}$. This finishes the argument for $\Phi=A_{n\geq 3}.$

For the symplectic group ${\rm Sp}_{2n}(R)$ for $n\geq 4$, we will use the calculations from our earlier paper \cite{explicit_strong_bound_sp_2n} split into four parts.  We will also first show the inclusion of sets as in (\ref{root_element_inclusion}) for short roots $\phi$ only and only later include long roots.

First, modulo conjugation with suitable Weyl group elements, we may assume that the sought after level generator $t$ of ${\rm Sp}_{2n}(R)$ is $b_{21},b_{n+1,1},b_{11}-b_{22}$ or $b_{11}-b_{n+1,n+1}$. For the first case note that we may using Lemma~\ref{stability_sln_prep_step1} assume that $b_{n,1}=0.$ Next, note that the proof of \cite[Lemma~2.7]{explicit_strong_bound_sp_2n} doesn't actually require $B$ to be in first Hessenberg form, but only for it to satisfy $b_{n,1}=0.$ Thus using \cite[Lemma~2.7]{explicit_strong_bound_sp_2n}, one obtains that there is an ideal $I$ in $R$ containing $t$ such that $I\subset\varepsilon_s(B,32)$.

For the second case $t=b_{n+1,1}$, we require some preliminary steps. Note first that combining the proof of Lemma~\ref{fixing_first_column_symplectic_step1} with the proof of the ideal stable range of Dedekind domains being at most $2$ from the proof of Lemma~\ref{stability_relative_ideal}, we may assume that the ideal in $R$ generated by the entries $b_{2,1},b_{3,1},\dots,b_{n,1}$ contains the entries $b_{n+2,1},\dots,b_{2n,1}.$ Then applying Lemma~\ref{stability_sln_prep_step1}, we may assume that the ideal $(b_{2,1},b_{3,1},\dots,b_{n,1},b_{n+2,1},\dots,b_{2n,1})$ is contained in $(b_{2,1},b_{3,1}).$ Hence applying the argument for the first case to the entries $b_{2,1}$ and $b_{3,1}$, we obtain that the ideal $I_1:=(b_{2,1},b_{3,1},\dots,b_{n,1},b_{n+2,1},\dots,b_{2n,1})$ is contained in $\varepsilon_s(B,2*32)=\varepsilon_s(B,64).$ But this process can also be applied to the $n+2.$th column of $B$ to find a second ideal $I_2\subset\varepsilon_s(B,64)$ with $(b_{1,n+2},b_{3,n+2},\dots,b_{n,n+2},b_{n+1,n+2},b_{n+3,n+2}\dots,b_{2n,n+2}).$ Taking $I_3:=I_1+I_2\subset\varepsilon_s(B,2*64)=\varepsilon_s(B,128)$, one obtains for $C:=B^{-1}$ that
\begin{align*}
B\equiv
\begin{pmatrix}
b_{11} & b_{12} & \cdots & b_{1,n} & b_{1,n+1} & 0 & b_{1,n+3} & \cdots & b_{1,2n}\\
0 & b_{22} & \cdots & b_{2,n} & b_{2,n+1} & b_{2,n+2} & b_{2,n+3} & \cdots & b_{2,2n}\\
0 & b_{32} & \cdots & b_{3,n} & b_{3,n+1} & 0 & b_{3,n+3} & \cdots & b_{3,2n}\\
\cdot & \cdot & \cdot & \cdot & \cdot & \cdot & \cdot & \cdot & \cdot\\
0 & b_{n,2} & \cdots & b_{nn} & b_{n,n+1} & 0 & b_{n,n+3} & \cdots & b_{n,2n}\\
b_{n+1,1} & b_{n+1,2} & \cdots & b_{n+1,n} & b_{n+1,n+1} & 0 & b_{n+1,n+3} & \cdots & b_{n+1,2n}\\
0 & b_{n+2,2} & \cdots & b_{n+2,n} & b_{n+2,n+1} & b_{n+2,n+2} & b_{n+2,n+3} & \cdots & b_{n+2,2n}\\
0 & b_{n+3,2} & \cdots & b_{n+3,n} & b_{n+3,n+1} & 0 & b_{n+3,n+3} & \cdots & b_{n+3,2n}\\
\cdot & \cdot & \cdot & \cdot & \cdot & \cdot & \cdot & \cdot & \cdot\\
0 & b_{2n,2} & \cdots & b_{2n,n} & b_{2n,n+1} & 0 & b_{2n,n+3} & \cdots & b_{2n,2n}
\end{pmatrix}\text{ mod }I_3
\end{align*}
and
\begin{align*}
C\equiv 
\begin{pmatrix}
c_{11} & c_{12} & c_{13} & \cdots & c_{1,n} & c_{1,n+1} & c_{1,n+2} & c_{1,n+3} & \cdots & c_{1,2n}\\
0 & c_{22} & 0 & \cdots & 0 & 0 & c_{2,n+2} & 0 & \cdots & 0\\
c_{31} & c_{32} & c_{33} & \cdots & c_{3,n} & c_{3,n+1} & c_{3,n+2} & c_{3,n+3} & \cdots & c_{3,2n}\\
\cdot & \cdot & \cdot & \cdot & \cdot & \cdot & \cdot & \cdot & \cdot & \cdot\\
c_{n,1} & c_{n,2} & c_{n,3} & \cdots & c_{nn} & c_{n,n+1} & c_{n,n+2} & c_{n,n+3} & \cdots & c_{n,2n}\\
c_{n+1,1} & 0 & 0 & \cdots & 0 & c_{n+1,n+1} & 0 & 0 & \cdots & 0\\
c_{n+2,1} & c_{n+2,2} & c_{n+2,3} & \cdots & c_{n+2,n} & c_{n+2,n+1} & c_{n+2,n+2} & c_{n+2,n+3} & \cdots & c_{n+2,2n}\\
c_{n+3,1} & c_{n+3,2} & c_{n+3,3} & \cdots & c_{n+3,n} & c_{n+3,n+1} & c_{n+3,n+2} & c_{n+3,n+3} & \cdots & c_{n+3,2n}\\
\cdot & \cdot & \cdot & \cdot & \cdot & \cdot & \cdot & \cdot & \cdot & \cdot\\
c_{2n,1} & c_{2n,2} & c_{2n,3} & \cdots & c_{2n,n} & c_{2n,n+1} & c_{2n,n+2} & c_{2n,n+3} & \cdots & c_{2n,2n}
\end{pmatrix}\text{ mod }I_3.
\end{align*}
But now as in the proof of \cite[Theorem~2.3]{explicit_strong_bound_sp_2n}, one obtains for the commutator
\[
B'=(B,I_{2n}+e_{12}-e_{n+2,n+1})
\]
that its $(n+2,1).$th entry is congruent to $b_{n+2,n+2}b_{n+1,1}$ and its $(1,2).$th entry congruent to $b_{n+2,n+2}b_{11}-1$ modulo the ideal $I_3.$ Hence applying the arguments of the first part to $B'$ and a suitable conjugate of $B'$, one can find two ideals $I_4,I_5$ both contained in $\varepsilon_s(B',32)\subset\varepsilon_s(B,64)$, such that $I:=I_5+I_4+I_3\subset\varepsilon_s(B,128+64+64)=\varepsilon_s(B,256)$ contains $b_{n+2,n+2}b_{n+1,1}$ and $b_{n+2,n+2}b_{11}-1$. Hence $I$ contains $t=b_{n+1,1}.$ This finishes the second case. Before going on, note that this ideal $I$ actually contains all off-diagonal entries of the first column of $B$ as well as $b_{n+2,n+2}b_{11}-1.$ For the case $t=b_{11}-b_{22}$, consider for $C:=B^{-1}$ the commutator
\begin{align*}
&B':=\left(B,I_{2n}+e_{12}-e_{n+2,n+1}\right)=(I_{2n}+Be_{12}B^{-1}-Be_{n+2,n+1}B^{-1})\cdot (I_{2n}-e_{12}+e_{n+2,n+1})\\
&=(I_{2n}+Be_{12}C-Be_{n+2,n+1}C)-(e_{12}-Be_{12}Ce_{12}+Be_{n+2,n+1}Ce_{12})\\
&\ \ \ +(e_{n+2,n+1}+Be_{12}Ce_{n+2,n+1}-Be_{n+2,n+1}Ce_{n+2,n+1})\\
&=(I_{2n}+Be_{12}C-Be_{n+2,n+1}C)-(e_{12}-c_{21}Be_{12}+c_{n+1,1}Be_{n+2,2})\\
&\ \ \ +(e_{n+2,n+1}+c_{2,n+2}Be_{1,n+1}-c_{n+1,n+2}Be_{n+2,n+1})
\end{align*}
Next, note that for the ideal $J$ generated by the off-diagonal entries of the first two columns of $B$, we obtain, as $J$ also contains the off-diagonal entries of the first two columns of $C,$ that:
\begin{align*}
&B'\equiv (I_{2n}+b_{11}e_{12}C-c_{n+1,n+1}Be_{n+2,n+1})-e_{12}+e_{n+2,n+1}+c_{2,n+2}Be_{1,n+1}\\
&\equiv I_{2n}+
\begin{pmatrix}
b_{11}c_{21} & b_{11}c_{22}-1 & b_{11}c_{23} & \cdots & b_{11}c_{2,n+1}\\
b_{21}c_{21} & b_{21}c_{22} & b_{21}c_{23} & \cdots & b_{21}c_{2,n+1}\\
\cdot & \cdot & \cdot & \cdot & \cdot\\
b_{n+1,1}c_{21} & b_{n+1,1}c_{22} & b_{n+1,1}c_{23} & \cdots & b_{n+1,1}c_{2,n+1}
\end{pmatrix}\text{ mod }J.
\end{align*}
Note that modulo $J$, the $(1,2).$th entry of $B'$ is $b_{11}c_{22}-1.$ But given that the off-diagonal entries of the first two columns of $C$ are $0$ modulo $J,$ one obtains $c_{22}\equiv b_{22}^{-1}\text{ mod }J$. Proceeding as in the second case, one obtains ideals $I_1,I_2\subset\varepsilon_s(B,256)$ such that $I_1+I_2\subset\varepsilon_s(B,2*256)=\varepsilon_s(B,512)$ contains $J$ and applying the calculations of the first case to a conjugate of the matrix $B'$, one can find an ideal $I_3$ such that $I:=I_1+I_2+I_3\subset\varepsilon_s(B,512+2*32)=\varepsilon_s(B,576)$ and $t=b_{11}-b_{22}$ is contained in $I.$ This finishes the third case. For the fourth case $t=b_{11}-b_{n+1,n+1},$ we proceed as follows: We start out with the ideal $I$ from the third case and consider the commutator and the following chain of equivalences modulo $I:$
\begin{align*}
&B'':=(B,I_{2n}+e_{1,n+2}+e_{2,n+1})=
(I_{2n}+Be_{1,n+2}B^{-1}+Be_{2,n+1}B^{-1})\cdot(I_{2n}-(e_{1,n+2}+e_{2,n+1}))\\
&\equiv (I_{2n}+b_{11}e_{1,n+2}C+b_{22}e_{2,n+1}C)\cdot (I_{2n}-(e_{1,n+2}+e_{2,n+1}))\\
&=(I_{2n}+b_{11}c_{n+2,n+2}e_{1,n+2}+b_{22}c_{n+1,n+1}e_{2,n+1})\cdot(I_{2n}-(e_{1,n+2}+e_{2,n+1}))\\
&\equiv (I_{2n}+b_{11}c_{n+2,n+2}(e_{1,n+2}+e_{2,n+1}))\cdot(I_{2n}-(e_{1,n+2}+e_{2,n+1}))\\
&\equiv I_{2n}+(b_{11}c_{n+2,n+2}-1)(e_{1,n+2}+e_{2,n+1})\\
&\equiv I_{2n}+(b_{11}b_{2,2}-1)(e_{1,n+2}+e_{2,n+1})\\
&\equiv I_{2n}+(b_{11}^2-1)(e_{1,n+2}+e_{2,n+1})\text{ mod }I
\end{align*}
So applying the argument from the first part to a suitable conjugate of $B''$, one can find an ideal $I_1\subset\varepsilon_s(B'',32)\subset\varepsilon_s(B,64)$ such that $b_{11}^2-1$ is contained in $I+I_1\subset\varepsilon_s(B,576+64)=\varepsilon_s(B,640).$ Note further that $b_{11}b_{n+1,n+1}-1\in I+I_1$ as $B$ is a symplectic matrix. Next, let $\C P$ be a prime ideal containing $I+I_1.$ Then $b_{11}^2-1=(b_{11}-1)\cdot (b_{11}+1)$ and so $b_{11}-1$ or $b_{11}+1$ is an element of $\C P.$ However, in both cases one obtains that $b_{n+1,n+1}-b_{11}\in\C P$ as $B$ is a symplectic matrix. This finishes the fourth case. 

But we note that as $t$ is an element of the radical of $I,$ there is some $k\in\mathbb{N}$ with $t^k\in I.$ However by the proof of \cite[Lemma~4.8(3)]{General_strong_bound}, one obtains than that $t^{2k}R\subset\varepsilon(B,3\cdot 640)=\varepsilon(B,1920)$ and so replacing the ideal $I$ by its subset $t^{2k}R$, we are done.
\end{proof}

\bibliographystyle{plain}
\bibliography{true_bound_arxiv_v3}

\def\polhk#1{\setbox0=\hbox{#1}{\ooalign{\hidewidth \lower1.5ex\hbox{`}\hidewidth\crcr\unhbox0}}} \def\polhk#1{\setbox0=\hbox{#1}{\ooalign{\hidewidth \lower1.5ex\hbox{`}\hidewidth\crcr\unhbox0}}} \def\polhk#1{\setbox0=\hbox{#1}{\ooalign{\hidewidth \lower1.5ex\hbox{`}\hidewidth\crcr\unhbox0}}} \def\polhk#1{\setbox0=\hbox{#1}{\ooalign{\hidewidth \lower1.5ex\hbox{`}\hidewidth\crcr\unhbox0}}} \def\polhk#1{\setbox0=\hbox{#1}{\ooalign{\hidewidth \lower1.5ex\hbox{`}\hidewidth\crcr\unhbox0}}} \def\polhk#1{\setbox0=\hbox{#1}{\ooalign{\hidewidth \lower1.5ex\hbox{`}\hidewidth\crcr\unhbox0}}} \def\cprime{$'$}
\begin{thebibliography}{10}

\bibitem{MR783068}
Z.~Arad, J.~Stavi, and M.~Herzog.
\newblock Powers and products of conjugacy classes in groups.
\newblock In {\em Products of conjugacy classes in groups}, volume 1112 of {\em Lecture Notes in Math.}, pages 6--51. Springer, Berlin, 1985.
\newblock DOI: {{https://doi.org/{10.1007/BFb0072286}}}.

\bibitem{avni_meiri_2019}
Nir Avni and Chen Meiri.
\newblock Words have bounded width in $\operatorname{SL}(n,\mathbb{Z})$.
\newblock {\em Compositio Mathematica}, 155(7):1245–1258, 2019.
\newblock DOI: {{https://doi.org/{10.1112/S0010437X19007334}}}.

\bibitem{avni2023conjugacy}
Nir Avni and Chen Meiri.
\newblock Conjugacy width in uniform higher rank arithmetic groups of orthogonal type, 2023.
\newblock arXiv: {https://arxiv.org/abs/2304.13173}.

\bibitem{MR244257}
H.~Bass, J.~Milnor, and J.-P. Serre.
\newblock Solution of the congruence subgroup problem for {${\rm SL}_{n}\,(n\geq 3)$} and {${\rm Sp}_{2n}\,(n\geq 2)$}.
\newblock {\em Inst. Hautes \'{E}tudes Sci. Publ. Math.}, (33):59--137, 1967.

\bibitem{MR0174604}
Hyman Bass.
\newblock {$K$}-theory and stable algebra.
\newblock {\em Inst. Hautes \'{E}tudes Sci. Publ. Math.}, (22):5--60, 1964.
\newblock DOI: {{https://doi.org/{10.1007/BF02684689}}}.

\bibitem{MR0249491}
Hyman Bass.
\newblock {\em Algebraic {$K$}-theory}.
\newblock W. A. Benjamin, Inc., New York-Amsterdam, 1968.

\bibitem{MR2509711}
Dmitri Burago, Sergei Ivanov, and Leonid Polterovich.
\newblock Conjugation-invariant norms on groups of geometric origin.
\newblock In {\em Groups of diffeomorphisms}, volume~52 of {\em Adv. Stud. Pure Math.}, pages 221--250. Math. Soc. Japan, Tokyo, 2008.

\bibitem{MR704220}
David Carter and Gordon Keller.
\newblock Bounded elementary generation of {${\rm SL}_{n}({\mathcal O})$}.
\newblock {\em Amer. J. Math.}, 105(3):673--687, 1983.
\newblock DOI: {{https://doi.org/{10.2307/2374319}}}.

\bibitem{ESTES1967343}
Dennis Estes and Jack Ohm.
\newblock Stable range in commutative rings.
\newblock {\em Journal of Algebra}, 7(3):343--362, 1967.
\newblock DOI: {{https://doi.org/{10.1016/0021-8693(67)90075-0}}}.

\bibitem{Gal-Kedra-Trost}
\'Swiatos{\l}aw~R. Gal, Jarek K\k{e}dra, and Alexander Trost.
\newblock Finite index subgroups in {C}hevalley groups are bounded: an addendum to "{O}n bi-invariant word metrics".
\newblock {\em Journal of Topology and Analysis}, 17(01):167--174, 2025.
\newblock DOI: {{https://doi.org/{10.1142/S1793525323500115}}}.

\bibitem{gvozdevsky2023width}
Pavel Gvozdevsky.
\newblock Width of {$SL(n,\mathscr{O}_S,I)$}.
\newblock {\em Communications in Algebra}, 51(4):1581--1593, 2023.
\newblock DOI: {{https://doi.org/{10.1080/00927872.2022.2139953}}}.

\bibitem{MR0396773}
James~E. Humphreys.
\newblock {\em Linear Algebraic Groups}, volume~21 of {\em Graduate Texts in Mathematics}.
\newblock Springer-Verlag, New York, 1975.
\newblock DOI: {{https://doi.org/{10.1007/978-1-4684-9443-3}}}.

\bibitem{Karni_paper}
S.~Karni.
\newblock Covering numbers of groups of small order and sporadic groups.
\newblock In {\em Products of conjugacy classes in groups}, volume 1112 of {\em Lecture Notes in Math.}, pages 52--196. Springer, Berlin, 1985.

\bibitem{KLM}
Jarek Kedra, Assaf Libman, and Ben Martin.
\newblock Strong and uniform boundedness of groups.
\newblock {\em Journal of Topology and Analysis}, 15(03):707--739, 2023.
\newblock DOI: {{https://doi.org/{10.1142/S1793525321500497}}}.

\bibitem{kunyavskiui2023bounded}
Boris Kunyavski{\u\i}, Eugene Plotkin, and Nikolai Vavilov.
\newblock Bounded generation and commutator width of {Chevalley} groups: function case.
\newblock {\em European Journal of Mathematics}, 9(53), 2023.
\newblock DOI: {{https://doi.org/{10.1007/s40879-023-00627-y}}}.

\bibitem{Kunyavskiı̆_Plotkin_Vavilov_2024}
Boris Kunyavski{\u{\i}}, Eugene Plotkin, and Nikolai Vavilov.
\newblock Uniform bounded elementary generation of {Chevalley} groups.
\newblock {\em Canadian Journal of Mathematics}, page 1–28, 2024.
\newblock DOI: {{https://doi.org/{10.4153/S0008414X24000713}}}.

\bibitem{MR1865975}
Martin~W. Liebeck and Aner Shalev.
\newblock Diameters of finite simple groups: sharp bounds and applications.
\newblock {\em Ann. of Math. (2)}, 154(2):383--406, 2001.
\newblock DOI: {{https://doi.org/{10.2307/3062101}}}.

\bibitem{MR0240214}
Hideya Matsumoto.
\newblock Sur les sous-groupes arithm\'etiques des groupes semi-simples d\'eploy\'es.
\newblock {\em Ann. Sci. \'Ecole Norm. Sup. (4)}, 2:1--62, 1969.
\newblock DOI: {{https://doi.org/{10.24033/asens.1174}}}.

\bibitem{MR3892969}
Aleksander~V. Morgan, Andrei~S. Rapinchuk, and Balasubramanian Sury.
\newblock Bounded generation of {$\rm SL_2$} over rings of {$S$}-integers with infinitely many units.
\newblock {\em Algebra \& Number Theory}, 12(8):1949--1974, 2018.
\newblock DOI: {{https://doi.org/{10.2140/ant.2018.12.1949}}}.

\bibitem{MR2357719}
Dave~Witte Morris.
\newblock Bounded generation of {${\rm SL}(n,A)$} (after {D}. {C}arter, {G}. {K}eller, and {E}. {P}aige).
\newblock {\em New York J. Math.}, 13:383--421, 2007.

\bibitem{MR1876657}
Michael Rosen.
\newblock {\em Number Theory in Function Fields}, volume 210 of {\em Graduate Texts in Mathematics}.
\newblock Springer-Verlag, New York, 2002.
\newblock DOI: {{https://doi.org/{10.1007/978-1-4757-6046-0}}}.

\bibitem{sinchuk2018decompositions}
Sergey Sinchuk and Andrei Smolensky.
\newblock Decompositions of congruence subgroups of {Chevalley} groups.
\newblock {\em International Journal of Algebra and Computation}, 28(06):935--958, 2018.

\bibitem{STEIN1971140}
Michael~R Stein.
\newblock Relativizing functors on rings and algebraic {K}-theory.
\newblock {\em Journal of Algebra}, 19(1):140--152, 1971.
\newblock DOI: {{https://doi.org/{10.1016/0021-8693(71)90123-2}}}.

\bibitem{197877}
Michael~R. Stein.
\newblock Stability theorems for {$K_{1}$}, {$K_{2}$} and related functors modeled on {C}hevalley groups.
\newblock {\em Japan. J. Math. (N.S.)}, 4(1):77--108, 1978.
\newblock DOI: {{https://doi.org/{10.4099/math1924.4.77}}}.

\bibitem{MR3616493}
Robert Steinberg.
\newblock {\em Lectures on {C}hevalley groups}, volume~66 of {\em University Lecture Series}.
\newblock American Mathematical Society, Providence, RI, 2016.
\newblock DOI: {{https://doi.org/{10.1090/ulect/066}}}.

\bibitem{stepanov1989stable}
Alexei~Vladimirovich Stepanov.
\newblock Stable range and stability of arbitrary vectors.
\newblock {\em Russian Mathematical Surveys}, 44(2):295--296, 1989.

\bibitem{MR1044049}
O.~I. Tavgen.
\newblock Bounded generation of {C}hevalley groups over rings of algebraic {$S$}-integers.
\newblock {\em Mathematics of the USSR-Izvestiya}, 36(1):101--128, 1991.
\newblock DOI: {{https://doi.org/{10.1070/IM1991v036n01ABEH001950}}}.

\bibitem{Chevalley_positive_char_tentative}
Alexander~A. Trost.
\newblock Bounded generation by root elements for {Chevalley} groups defined over rings of integers of function fields with an application in strong boundedness, 2021.
\newblock arXiv: {https://arxiv.org/abs/2108.12254}.

\bibitem{explicit_strong_bound_sp_2n}
Alexander~A Trost.
\newblock Explicit strong boundedness for higher rank symplectic groups.
\newblock {\em Journal of Algebra}, 604:694--726, 2022.
\newblock DOI: {{https://doi.org/{10.1016/j.jalgebra.2022.03.041}}}.

\bibitem{General_strong_bound}
Alexander~A Trost.
\newblock Strong boundedness of split {Chevalley groups}.
\newblock {\em Israel Journal of Mathematics}, 252:1--46, 2022.
\newblock DOI: {{https://doi.org/{10.1007/s11856-022-2344-0}}}.

\bibitem{https://doi.org/10.1112/blms.12925}
Alexander~A. Trost.
\newblock Elementary bounded generation for {${\rm SL}_n$} for global function fields and $n\geqslant 3$.
\newblock {\em Bulletin of the London Mathematical Society}, 56(1):219--231, 2024.
\newblock DOI: {{https://doi.org/{10.1112/blms.12925}}}.

\bibitem{explicit_strong_bound_sp4_pseudo_good}
Alexander~Alois Trost.
\newblock Bounded generation for congruence subgroups of {${\rm Sp}_4(R)$}.
\newblock {\em Journal of Algebra and Its Applications}, 22(08), 2023.
\newblock DOI: {{https://doi.org/{10.1142/S0219498823501748}}}.

\bibitem{MR2822515}
NA~Vavilov, AV~Smolensky, and B~Sury.
\newblock Unitriangular factorizations of {Chevalley} groups.
\newblock {\em Journal of Mathematical Sciences}, 183(5):584--599, 2012.
\newblock DOI: {{https://doi.org/{10.1007/s10958-012-0826-z}}}.

\end{thebibliography}

\end{document}